\documentclass[11pt]{article}

\usepackage{amsfonts,amssymb,graphicx}
\usepackage{amsfonts}
\usepackage{amsmath}
\usepackage{amstext}
\usepackage{amsthm}
\usepackage{xspace}
\usepackage{fullpage}

\newcommand{\Z}{\mathbb Z}
\newcommand{\N}{\mathbb N}
\newcommand{\C}{\mathbb C}

\newcommand{\R}{\mathbb R}

\newcommand{\hp}{\mathbb H}

\renewcommand{\Re}{\mathrm{Re}}
\renewcommand{\Im}{\mathrm Im}

\newcommand{\set}[1]{\left\{#1\right\}}
\newcommand{\eps}{\varepsilon}

\newcommand\BB{{\cal B}}

\newcommand\FF{{\cal F}}

\newcommand\HH{{\cal H}}

\newcommand\JJ{{\cal J}}
\newcommand\KK{{\cal K}}
\newcommand\LL{{\cal L}}

\newcommand\OO{{\cal O}}
\newcommand\PP{{\cal P}}
\newcommand\QQ{{\cal Q}}
\newcommand\RR{{\cal R}}

\newcommand\ZZ{{\cal Z}}

\newtheorem {theorem}{Theorem}[section]

\newtheorem {lemma}[theorem]{Lemma}
\newtheorem {proposition}[theorem]{Proposition}
\newtheorem {corollary}[theorem]{Corollary}

\theoremstyle{remark}
\newtheorem {remark}[theorem]{Remark}
\newtheorem {remarks}[theorem]{Remarks}
\theoremstyle{definition}
\newtheorem{definition}[theorem]{Definition}
\numberwithin{equation}{section}

\begin{document}

\title{A thermodynamic approach to two-variable Ruelle and Selberg zeta functions via the Farey map}

\author{Claudio Bonanno
\thanks{(corresponding author) Dipartimento di Matematica Applicata,
Universit\`a di Pisa, via F. Buonarroti 1/c, I-56127 Pisa, Italy,
email: $<$bonanno@mail.dm.unipi.it$>$} \and Stefano Isola
\thanks{Dipartimento di Matematica e Informatica, Universit\`a
di Camerino, via Madonna delle Carceri, I-62032 Camerino, Italy.
e-mail: $<$stefano.isola@unicam.it$>$}}
\date{}
\maketitle

\begin{abstract}
In this paper we consider the transfer operator approach to the Ruelle and Selberg zeta functions associated to continued fractions transformations and the geodesic flow on the full modular surface. We extend the results by Ruelle and Mayer to two-variable zeta functions, $\zeta(q,z)$ and $Z(q,z)$. The $q$ variable plays the role of the inverse temperature and the introduction of the ``geometric variable'' $z$ is essential in the tentative to provide a general approach, based on the Farey map, to the correspondence between the analytic properties of the zeta functions themselves, the spectral properties of a class of generalised transfer operators and the theory of a generalisation of the three-term functional equations studied by Lewis and Zagier. The first step in this direction is a detailed study of the spectral properties of a family of signed transfer operators $\PP^{\pm}_{q}$ associated to the Farey map.
\end{abstract}


\section{Introduction}

The transfer operator approach to the Selberg zeta function for the full modular group $PSL(2,\Z)$ introduced by Mayer in \cite{Ma3} led to new interesting interactions between number theory and the thermodynamic approach to dynamical systems. The correspondence between the zeroes of the Selberg zeta function and the eigenfunctions of the transfer operators for the Gauss map has been  subsequently studied in particular in connection with the theory of period functions for cusp and non-cusp forms on $PSL(2,\Z)$ in \cite{CM1,CM2,LeZa}. In this paper we extend the approach in \cite {Ma3} to signed transfer operators $\PP^{\pm}_{q}$ for the Farey map, which is connected to the Gauss map by an induction procedure. This approach clarifies some aspects of the results in \cite{Ma3}, and we are naturally led to the definitions of two-variable Ruelle and Selberg zeta functions for which extensions of Mayer results hold.

We now discuss in details the content of the paper. The first issue is the study of the signed transfer operators $\PP^{\pm}_{q}$ for the Farey map defined in (\ref{farey}). The spectral properties of transfer operators for uniformly expanding maps of an interval are now well understood and depend crucially on the Banach space considered \cite{Bal}. For sufficiently regular functions it turns out that the transfer operator is quasi-compact, hence the spectrum is made of isolated eigenvalues with finite multiplicity and the essential part, a disk of radius strictly smaller than the spectral radius. However the essential spectral radius depends on the expanding constant $\rho$ of the map and on the degree of regularity of the functions. In particular as $\rho \to 1$ from above, the essential spectral radius converges to the spectral radius. The Farey map $F$ is a prototype of a smooth intermittent map on the unit interval $[0,1]$, being expanding everywhere except at the origin, a neutral fixed point. Hence for the Farey map $\rho =1$, and classical approaches to the spectral properties of its transfer operator fail. As a matter of fact, using {\sl ad hoc} techniques, the spectrum of the Farey transfer operator when acting on suitable spaces of holomorphic functions has been shown to have an absolutely continuous component given by the unit interval (see \cite{Rugh,Is, BGI}). 

Here we consider the family  of signed generalised transfer operators $\PP_{q}^{\pm}$ associated to $F$ and defined in (\ref{transfer-farey}) and provide a detailed study of their spectral properties on a space of holomorphic functions on an open domain containing $(0,1)$. In Section \ref{inizio}, we prove our first main result which is a complete characterisation of the eigenfunctions of $\PP_{q}^{\pm}$ with eigenvalues not embedded in the continuous spectrum (Theorem \ref{eigenfunc-acca}). This is obtained in terms of an integral transform defined on weighted spaces $L^{p}((0,+\infty),m_{q}(t))$ with $dm_{q}(t) = t^{2\Re(q)-1} e^{-t}$. A similar approach has been used in \cite{Is,Pre, GI, BGI}. In \cite{BGI} the Hilbert space $L^{2}((0,+\infty),m_{q}(t))$ has been studied in some detail and the resulting issues are used in this paper. In particular, a simple argument (Proposition \ref{proprieta-farey}) shows that eigenfunctions of $\PP_{q}^{\pm}$ satisfy the three-term functional equation (\ref{lewis}), which for $\lambda=1$ is but the Lewis functional equation studied in \cite{LeZa}, where a class of solutions of this equation is proved to be in one-to-one correspondence with the Maass cuspidal and non-cuspidal forms on the full modular group $PSL(2,\Z)$. Some of our results, in particular Corollary \ref{autofunz}, are then generalisations of results in \cite{LeZa} to eigenvalues $\lambda \not= 1$.
  
In the second part, we use an inducing procedure for $F$ which was first introduced in \cite{PS} for a general class of intermittent interval maps. The idea is to consider an induced map $G$ on the interval with respect to the first passage time at a subset of $[0,1]$ away from the neutral fixed point, and to study the spectral properties of the transfer operators $\QQ$ of $G$. Then functional relations between $\QQ$ and $\PP$ allow to translate the spectral properties of $\QQ$ into those of $\PP$ and viceversa. For applications of this method see also \cite{Rugh,Is,Pre}.
It is well known that the Farey map is related to number theory, and in particular to the continued fractions expansion of $x\in [0,1]$. Moreover, if we define $G$ to be the induced map with respect to the first passage time at the interval $\left( \frac 12, 1 \right]$, it turns out that $G$ is the Gauss map on the unit interval, see (\ref{indg}). In Section \ref{induced-op}, following this procedure, we define in (\ref{q}) a two-parameter family of transfer operators $\QQ_{q,z}$ associated to $G$ and look for functional relations between $\QQ_{q,z}$ and $\PP_{q}^{\pm}$. This is obtained in Theorem \ref{q-vs-p} and allows us to obtain a one-to-one correspondence between all eigenfunctions to the eigenvalues $\pm 1$ of $\QQ_{q,z}$ and some eigenfunctions with eigenvalue $\frac 1 z$ of $\PP_{q}^{\pm}$ (Corollary \ref{relaz-autof}). 

Let us recall that the properties of the transfer operators $\QQ_{q,1}$ have been already studied in \cite{Ma1,Ma2,Ma3,CM2}, and also discussed in \cite{LeZa}. In particular, in \cite{Ma3} it is proved that the eigenfunctions with eigenvalues $\pm 1$ of $\QQ_{q,1}$ are in one-to-one correspondence with the zeroes of the Selberg zeta function for the full modular group $PSL(2,\Z)$. In turn, it is the content of Selberg trace formula that the zeroes of the Selberg zeta function are in one-to-one correspondence with the Maass cusp forms and the non-trivial zeroes of the Riemann zeta function, which are related to the Maass non-cuspidal forms. One then obtains a relation between eigenfunctions to the eigenvalues $\pm 1$ of $\QQ_{q,1}$ and a class of solutions of the Lewis three-term functional equation. This is discussed in \cite{LeZa} and proved without appealing to the Selberg zeta function. 

It is one of the aims of this paper to obtain the latter relation from the eigenfunctions of the transfer operators $\PP_{q}^{\pm}$, thus giving, in the spirit of \cite{Ma3}, a thermodynamic formalism approach to the period functions studied in \cite{LeZa}. The first step is Theorem \ref{q-vs-p} mentioned above. Notice in particular the term $\pm c\mu ^{x}$ in (\ref{ug-funz}). This term is indeed responsible for the restriction to a class of eigenfunctions of $\PP_{q}^{\pm}$ in Corollary \ref{relaz-autof}, i.e. to a class of solutions of theLewis functional equation. The second step is the main result of Section \ref{sect-zeta}, Theorem \ref{main-zeta}. We first show that the transfer operators $\QQ_{q,z}$ are of trace class on a suitable Banach space, hence the Fredholm determinants $\det(1\mp \QQ_{q,z})$ are well defined, and the traces can be computed as in \cite{Ma1} and \cite{Is}. Hence we obtain (\ref{due-zeta}), which is contained in \cite{Ma3} for $z=1$, and leads to a definition of a two-variable Selberg zeta function $Z(q,z)$. Altogether these steps show that the zeroes of the function $Z(q,z)$ are in one-to-one correspondence with the eigenfunctions of $\PP_{q}^{\pm}$ with eigenvalue $\frac 1z$ when the constant $c$ which appears in (\ref{ug-funz}) (see also Corollary \ref{autofunz}) vanishes. In the case $z=1$, as a corollary of our results we also obtain a new proof for the characterisation of the nontrivial zeroes of the function $Z(q,1)$ in terms of even and odd cusp and non-cusp forms given in \cite{efrat}. This is the content of Theorem \ref{efrat} in Section \ref{z=1}.

In the same Section we also discuss a two-variable Ruelle zeta function $\zeta(q,z)$ for the Farey map, defined in (\ref{zeta}). Our main result is equation (\ref{R-zeta}) in Theorem \ref{main-zeta}, which gives an expression of $\zeta(q,z)$ in terms of the Fredholm determinants $\det(1\mp \QQ_{q,z})$, extending a previous result in \cite{Is}. For a discussion of multi-variable zeta functions in dynamical systems and number theory see \cite{Lag}.

Finally in Section \ref{sect-ff} we come back to the Farey map and its relations to number theory. Using a formal manipulation of (\ref{due-zeta}) and (\ref{R-zeta}), which is similar to the approach in \cite{Rugh}, we obtain an expression for the zeta functions $Z(q,z)$ and $\zeta(q,z)$ as exponentials of power series whose coefficients are obtained as sums along branches of the Farey tree (Theorem \ref{formula-selberg}).

\section{Transfer operators for the Farey map} \label{inizio}
\noindent
Let $F:[0,1]\to [0,1]$ be  the {\sl Farey map} defined by
\begin{equation} \label{farey}
F(x)=\left\{
\begin{array}{ll}
\frac{x}{1-x} & \mbox{if }\ 0\le x\le \frac{1}{2}\\[0.3cm]
\frac{1-x}{x} & \mbox{if }\ \frac{1}{2} \le x \le 1
\end{array} \right.
\end{equation}

\noindent To this map we can associate a family of \emph{signed
generalized transfer operators} $\PP_q^{\pm}$, with $q\in C$, $\Re(q)>0$, whose action on a function $f(x):[0,1]\to \C$ is given by a weighted sum over the values of $f$ on
the set $F^{-1}(x)$, namely if
\begin{equation} \label{p0}
(\PP_{0,q} f)(x) := \left(\frac{1}{x+1}\right)^{2q} \ f\left(
\frac{x}{x+1}\right)
\end{equation}
\begin{equation} \label{p1}
(\PP_{1,q} f)(x) := \left(\frac{1}{x+1}\right)^{2q} \ f\left(
\frac{1}{x+1}\right)
\end{equation}
\noindent
we define
\begin{equation} \label{transfer-farey}
(\PP_q^{\pm} f)(x) := (\PP_{0,q} f)(x) \pm (\PP_{1,q} f)(x)
\end{equation}

\noindent Since the operators $\PP_{q}^{\pm}$ are defined by multiplication and composition with real analytic maps which extend to holomorphic maps on a neighbourhood of the interval $[0,1]$, it is natural to consider the action of $\PP_{q}^{\pm}$ on the space $\HH(B)$ of holomorphic functions on the open domain 
$$
B:= \set{x\in \C\ :\ \left| x-\frac 1 2 \right| < \frac 1 2}.
$$ 
We point out that we still denote by $x$ the
complex variable. We are interested in the spectral properties of
these transfer operators, hence first of all we give a
characterisation of their eigenfunctions. To this aim we give some
properties of the action of $\PP_{q}^{\pm}$ on $\HH(B)$. It is
useful to consider also the involution $\JJ_q$ with
\begin{equation} \label{involution}
(\JJ_q f)(x) := \frac{1}{x^{2q}}\ f\left(\frac{1}{x}\right)
\end{equation}
By definition of involution, any function $f\in
\HH(\set{\Re(x)>0})$ can be written as a sum of eigenfunctions of
$\JJ_{q}$, that is $f = \varphi_{+} + \varphi_{-}$ with $\JJ_{q}
\varphi_{+} = \varphi_{+}$ and $\JJ_{q} \varphi_{-} = -
\varphi_{-}$.

\begin{proposition} \label{proprieta-farey}
(i) If $f \in \HH(B)$ then $\PP_{q}^{\pm} f \in \HH(\set{\Re(x)>0})$;

\noindent (ii) if $f \in \HH(B)$ is an eigenfunction of
$\PP_{q}^{\pm}$ with eigenvalue $\lambda \in \C \setminus
\set{0}$, then $f \in \HH(\set{\Re(x)>0})$ and $\JJ_q f = \pm f$. In particular if $\JJ_{q}f = -f$ then $f(1)=0$. Moreover
\begin{equation} \label{lewis}
\lambda f(x) - f(x+1) = \left( \frac{1}{x+1} \right)^{2q} \
f\left( \frac{x}{x+1} \right) \qquad \Re(x)>0\, ;
\end{equation}

\noindent (iii) if $f\in \HH(\set{\Re(x)>0})$ satisfies
(\ref{lewis}) with $\lambda \in \C \setminus \set{0}$, then
$\varphi_{\pm} := \frac 1 2 (f \pm \JJ_{q}f)$ satisfies $\PP_{q}^{\pm}
\varphi_{\pm} = \lambda \varphi_{\pm}$.
\end{proposition}

\noindent
\begin{proof}
(i) follows simply by the fact that if $\Re(x) >0$ then both
$\frac{x}{x+1}$ and $\frac{1}{x+1}$ are in $B$.

\noindent (ii) The first assertion follows by (i). For the second,
by the definition of $\JJ_{q}$ given in (\ref{involution}) one
easily checks that
$$\JJ_q  \PP_q^\pm f =\pm  \, \PP_q^\pm f.$$
Hence if $f$ is an eigenfunction we can rewrite the previous
expression substituting $\PP_q^\pm f$ with $\lambda f$ and the
second assertion follows. Finally (\ref{lewis}) is obtained by
using the fact that for eigenfunctions it holds
$$\pm (\PP_{1,q} f)(x) = \pm\, (\JJ_{q} f)(x+1) = f(x+1).$$

\noindent (iii) Let $\varphi_{+}$ and $\varphi_{-}$ be defined as
above. They satisfy $f=\varphi_{+} + \varphi_{-}$ and $\JJ_{q}
\varphi_{\pm} = \pm \varphi_{\pm}$. Moreover, one can easily
checks that if $f$ satisfies (\ref{lewis}) with $\lambda \in \C
\setminus \set{0}$ than the same holds true for $\JJ_q f$. Hence
also $\varphi_{\pm}$ satisfy (\ref{lewis}) with the same
$\lambda$, and their invariance properties under $\JJ_{q}$ imply
the assertion.
\end{proof}

\begin{remark}
Eq. (\ref{lewis}) has been studied in \cite{Le,LeZa} for $\lambda=1$ in connection with Maass forms on the full modular group $PSL(2,\Z)$.
Moreover part (iii) of the proposition implies that eq.
(\ref{lewis}) has no solutions for $|\lambda|$ bigger than both
the spectral radii of $\PP_{q}^{\pm}$.
\end{remark}
\noindent
We now introduce the functional analytic settings. In this section we are interested in defining a family of spaces $\HH_{q,\mu}$ to which the
eigenfunctions of $\PP_{q}^{\pm}$ belong. 

Throughout this paper we use the notation $\xi:=\Re(q)$ and are assuming $\xi >0$. Let $\LL$ and
$\LL^{-1}$ denote the Laplace and inverse Laplace transform,
respectively. We recall that
\begin{equation} \label{power-trans}
\LL[t_+^{\nu-1}](x) = \Gamma(\nu)\ x^{-\nu}  \qquad
\LL^{-1}[x^{-\nu}](t) = \frac{t_+^{\nu-1}}{\Gamma(\nu)} \qquad t
\in \R, \ \nu \in \C \setminus \Z^-
\end{equation}
where $\Gamma(\nu)$ is the usual Gamma function, $t_+^\alpha:=
t^\alpha H(t)$, with $H(t)$ the Heaviside function, and for
$\Re(\nu)<0$ we are working with generalised functions. In the following when it causes no confusion we drop for brevity the independent variable in the Laplace and inverse Laplace transforms. Hence we use the notation
\begin{equation} \label{notation}
\LL[\phi(t)] := \LL[\phi(t)](x)\qquad \LL^{-1}[f(x)]:= \LL^{-1}[f(x)](t)
\end{equation}
We now consider the integral transform $\BB_{q}$ introduced in
\cite{Is} and defined as
\begin{equation} \label{borel-def}
\phi(t) \mapsto {\BB}_q [\phi](x):= \frac{1}{x^{2q}}\ \int_0^\infty
e^{-\frac{t}{x}}\, e^t\, \phi (t)\, m_q(dt)
\end{equation}
where $m_{q}$ is the absolutely continuous measure on $\R^{+}$
defined as $m_q(dt)= t^{2q-1}\, e^{-t}\, dt$. As for the Laplace transform in the notation we drop the dependence on the variable.

Finally for each $q$ with $\xi:= \Re(q)>0$ we introduce the notation
\begin{equation}\label{spazi-lpq}
L^{p}(m_{q}) := \set{\phi :\R^{+}\to \C\ :\ \int_{0}^{\infty}\, |\phi(t)|^{p}\, t^{2\xi -1} e^{-t}\, dt <\infty}
\end{equation}
with the norm
$$
\| \phi \|_{p} := \left( \int_{0}^{\infty}\, |\phi(t)|^{p}\, t^{2\xi -1} e^{-t}\, dt \right)^{\frac 1p}
$$
It is immediate to check that
$$
L^{1}(m_{q}) \ni \phi \mapsto \BB_{q} [\phi] \in \HH(B)
$$
and that $\BB_{q}$ is continuous on $L^{1}(m_{q})$ with values on $\HH(B)$ with the standard topology induced by the family of supremum norms on compact subsets of $B$. Indeed
$$
| \BB_{q} [\phi] (x) | \le |x|^{-2\xi} \, \| \phi\|_{1}\qquad \forall\, x\in B
$$
hence $\BB_{q}$ sends bounded sets of $L^{1}(m_{q})$ into bounded sets of $\HH(B)$.
Moreover, since $m_{\xi}(0,\infty) = \Gamma(2\xi)$, one has $L^{p}(m_{q}) \subset L^{1}(m_{q})$ for all $p\in [1,\infty]$. We now study the spaces which contain the eigenfunctions of $\PP_{q}^{\pm}$ in terms of the $\BB_{q}$ transform.

\begin{definition} \label{spazi-acca}
For $\Re(q)>0, \, q\not= \frac 1 2$, $\mu\in \C$ and $p \in [1,+\infty]$, let $\HH_{q,\mu}^{p}$ be the space of functions $f(x)$ written as
$$
f(x) = \frac{c\, \mu^{\frac 1 x}}{x^{2q}} + \BB_{q} \left[ \frac b
t + \phi(t) \right](x) \qquad x\in B
$$
for $c, b \in \C$ and $\phi \in L^{p}(m_{q})$ as defined in \eqref{spazi-lpq}.
\end{definition}

\begin{remarks} \label{osservazioni-su-hq}
The definition of the spaces
$\HH_{q,\mu}^{p}$ needs some comments. First of all to prove that Definition
\ref{spazi-acca} is well posed, we extend the definition of the
$\BB_q$ transform by means of (\ref{power-trans}) to the function $\phi(t) = \frac 1t$ as
\begin{equation} \label{bq-su-t}
\BB_q \left[ \frac{1}{t} \right] = \frac{1}{x^{2q}}\
\int_{0}^{\infty} \ e^{- \frac t x} \frac{1}{t}\, t^{2q-1}  \, dt
= \frac{1}{x^{2q}}\ \LL[t^{2q-2}]\left(\frac 1 x\right) = \Gamma(2q-1)\,
x^{-1}
\end{equation}
which is well defined for $\xi=\Re(q)>0$ and $q\not= \frac 1 2$, whereas the first integral is defined only for $\xi> \frac 1 2$. Hence, the condition $q\not= \frac 12$ comes from the term $\frac 1 t$.
\noindent
A typical example of functions to which we will apply the $\BB_{q}$ transform is
\begin{equation} \label{typical}
\frac b t + \phi (t) = \frac{e^{-t}}{1-e^{-t}}\
\frac{a_0}{\Gamma(2q)} + \frac{e^{-t}}{1-e^{-t}}\ \sum_{n\ge 1}\
\frac{a_{n}}{\Gamma(n+2q)} \ t^{n} \qquad t\in \R^{+}
\end{equation}
with $\limsup_{n} (a_{n})^{\frac 1 n} \le 1$, for which
$b=\frac{a_0}{\Gamma(2q)}$ and $\phi$ is in $L^2(m_q)$ with
$\phi(0) = \frac{a_1}{\Gamma(2q+1)} - \frac{a_0}{2\, \Gamma(2q)}$.
\noindent
A further consideration is about the function to which we apply
$\BB_{q}$. This function is written as $\left( \frac bt + \phi(t) \right)$ above, but could be written in many different ways, and actually some of them will be used below. However the main feature of this term that we want to stress is that it can be divided into
two parts with respect to the behaviour at the origin: the first
part has a singularity at $t=0^{+}$ of order $t^{-1}$, hence in particular does not belong to $L^{p}(m_{q})$ for $\xi\le \frac p2$; the second part is instead in $L^{p}(m_{q})$.
\end{remarks}

\begin{proposition} \label{azione-su-spazi}
The transfer operators $\PP_q^\pm$ leave invariant the spaces
$\HH_{q,\mu}^{2}$ for any $\mu \in \C$. In particular
\begin{equation} \label{p-s0-hq}
\begin{array}{c}
\PP_{0,q} \left( \frac{c\, \mu^{\frac 1 x}}{x^{2q}} + \BB_{q}
\left[  \psi \right] \right) = \frac{c\mu \, \mu^{\frac 1
x}}{x^{2q}} + \BB_{q} \left[ M(\psi) \right] \\[0.3cm]
\PP_{1,q} \left( \frac{c\, \mu^{\frac 1 x}}{x^{2q}} + \BB_{q}
\left[  \psi \right] \right) = \BB_q \left[ c\, \bar \phi_{q,\mu}
+ N_q(\psi) \right]
\end{array}
\end{equation}
where $M$ and $N_q$ are linear operators defined by
\begin{equation} \label{emme}
M (\psi) (t) = e^{-t}\, \psi(t)
\end{equation}
\begin{equation} \label{enne}
N_q (\psi) (t) = \int_0^\infty\
\frac{J_{2q-1}(2\sqrt{st})}{(st)^{q-\frac 1 2}}\, \psi(s)\,
m_q(ds)
\end{equation}
with $J_p$ the Bessel function of order $p$, and the function
$\bar \phi_{q,\mu}(t)$ defined by
\begin{equation} \label{funzione-bar}
\bar \phi_{q,\mu} (t) := \mu \ \sum_{k=0}^\infty\ \frac{(\log
\mu)^k}{\Gamma(k+1)\, \Gamma(k+2q)}\, t^k
\end{equation}
letting $\bar \phi_{q,0}(t) \equiv 0$.
\end{proposition}

\begin{proof}
To prove (\ref{p-s0-hq}), we study the action of the transfer
operators on the different terms of a function $f \in
\HH_{q,\mu}^{2}$, for $q$ and $\mu$ fixed.

\noindent
First of all, it is immediate that $\PP_{0,q}\left(\frac{
\mu^{\frac 1 x}}{x^{2q}} \right) = \mu\, \frac{ \mu^{\frac 1
x}}{x^{2q}}$ and $\PP_{1,q}\left(\frac{ \mu^{\frac 1 x}}{x^{2q}}
\right) = \mu^{x+1}$. It follows that the function $\bar
\phi_{q,\mu}$ we are looking for has to satisfy
$$
\mu^{x+1} = \BB_q[\bar \phi_{q,\mu}] = \frac{1}{x^{2q}} \, \LL
[t_+^{2q-1} \bar \phi_{q,\mu}(t)]\left(\frac 1 x\right)
$$
hence
$$
\bar \phi_{q,\mu}(t) = \frac{1}{t_+^{2q-1}}\, \LL^{-1} \left[
\frac{\mu^{1+\frac 1 x}}{x^{2q}} \right] (t)
$$
Expression (\ref{funzione-bar}) follows by a straightforward
computation. It remains to prove that the function on the right
hand side of (\ref{funzione-bar}) admits a $\BB_q$ transform. If
we differentiate term by term, it follows that for any $n\in \N$
there exists a positive polynomial $p_n(t)$ of degree $n$ such
that
$$
\frac{d}{dt}\, \bar \phi_{q,\mu} (t) \le p_n(t) + \frac{|\log
\mu|}{n+2q}\ |\bar \phi_{q,\mu}(t)|
$$
From this it follows that the behaviour of $\bar \phi_{q,\mu}$ at
infinity is slower than $e^{\eps t}$ for all $\eps >0$. Moreover
$\bar \phi_{q,\mu}(0) = \frac{\mu}{\Gamma(2q)}$. Hence actually
$\bar \phi_{q,\mu} \in L^2(m_q)$.

\noindent
For $f \in \HH_{q,\mu}^{2}$ we write $f(x) = \frac{c\, \mu^{\frac 1
x}}{x^{2q}} + \BB_q [\psi]$, where $\psi(t) = \frac b t +
\phi(t)$. Expressions (\ref{p-s0-hq}) have been proved in \cite[Proposition 2.3]{DEIK}
for functions $\phi$ in $L^2(m_q)$. It remains to prove that the
same holds for the term $\frac b t$. For $\PP_{0,q}$ we simply
have
$$
\PP_{0,q} \left( \BB_q \left[ \frac 1 t\right] \right) = \frac{1}{x^{2q}}\
\int_{0}^{\infty} \ e^{- \frac t x}\, \frac{e^{-t}}{t}\, t^{2q-1}
\, dt = \BB_q \left[ \frac{e^{-t}}{t} \right]
$$
For $\PP_{1,q}$ we have to prove
$$
\PP_{1,q} \left( \BB_q \left[ \frac 1 t \right] \right) = \BB_q
\left[ N_q \left( \frac 1 t \right) \right]
$$
Using the power series expansion
$$
J_p(x) = \frac{x^p}{2^p}\ \sum_{m=0}^\infty\, \frac{(-1)^m\,
x^{2m}}{2^{2m}\, m!\, \Gamma(m+p+1)}
$$
for the Bessel function $J_{2q-1}$, we have
$$
N_q \left( \frac 1 t \right) = \int_0^\infty\ \sum_{m=0}^\infty\,
\frac{(-1)^m\, t^m}{m!\, \Gamma(m+2q)} \ s^{m+2q-2}\, e^{-s}\ ds =
\sum_{m=0}^\infty\, \frac{(-1)^m\, t^m}{m!\, \Gamma(m+2q)}\
\LL[s^{m+2q-2}](1) =
$$
$$
=\sum_{m=0}^\infty\, \frac{(-1)^m\, \Gamma(m+2q-1)}{m!\,
\Gamma(m+2q)}\ t^m
$$
It follows that
$$
\BB_q \left[ N_q \left( \frac 1 t \right) \right] =
\frac{1}{x^{2q}}\, \LL \left[ t_+^{2q-1} N_q \left( \frac 1 t
\right) \right]\left(\frac 1 x\right) = \frac{1}{x^{2q}}\,
\sum_{m=0}^\infty\, \frac{(-1)^m\, \Gamma(m+2q-1)}{m!\,
\Gamma(m+2q)}\ \LL \left[ t_+^{m+2q-1} \right]\left(\frac 1 x\right)=
$$
$$
=\sum_{m=0}^\infty\, \frac{(-1)^m\, \Gamma(m+2q-1)}{m!}\ x^m =
\Gamma(2q-1)\ (1+x)^{1-2q}
$$
Moreover,
$$
\PP_{1,q} \left( \BB_q \left[ \frac 1 t \right] \right) =
\int_0^\infty\ e^{-tx}\, e^{-t}\, t^{2q-2}\, dt = \LL \left[
e^{-t}\, t^{2q-2} \right] = \Gamma(2q-1)\, (1+x)^{1-2q}
$$
This completes the proof.
\end{proof}

\begin{remark} \label{immagin-p1}
Notice that the $\PP_{1,q}(\HH_{q,\mu}^{2})$ is contained in the set
of functions which are $\BB_q$ transform of functions in
$L^2(m_q)$. This follows from the fact that $\bar \phi_{q,\mu} \in
L^2(m_q)$, $N_q(\phi) \in L^2(m_q)$ for any $\phi \in L^2(m_q)$
(see \cite{BGI}), and writing
\begin{equation} \label{nq-1t} 
N_q  \left( \frac 1 t \right) = \sum_{m=0}^\infty\,
\frac{(-1)^m}{m!\, (m+2q-1)}\ t^m = \frac{\int_0^t\
s^{2q-2}\, e^{-s}\ ds}{t^{2q-1}} = \frac{\Gamma(2q-1) - \int_{t}^{\infty}\ s^{2q-2}\, e^{-s}\ ds}{t^{2q-1}}
\end{equation}
where the third term makes sense only for $\xi> \frac 1 2$ and the last one for all $\xi>0$ and $t>0$. In particular it follows from (\ref{nq-1t}) that $N_q  \left( \frac 1 t \right)$ is bounded as $t\to 0^{+}$, and behaves like $O(t^{1-2q})$ as $t\to +\infty$. Moreover using the Dirac delta function $\delta_{-\log \mu}$ we can write
$$
\frac{\mu^{\frac 1 x}}{x^{2q}} = \BB_q \left[ \frac{\delta_{-\log \mu}(t)}{t^{2q-1}} \right]\qquad \qquad \frac{\mu \, \mu^{\frac 1
x}}{x^{2q}} = \BB_q \left[ M\left( \frac{\delta_{-\log \mu}(t)}{t^{2q-1}} \right) \right]
$$
and
$$
\bar \phi_{q,\mu} (t) = \mu\ \frac{J_{2q-1} (2 \sqrt{-t\, \log
\mu})}{(\sqrt{-t\, \log \mu})^{q- \frac 1 2}} = N_q \left(
\frac{\delta_{-\log \mu}(t)}{t^{2q-1}} \right)
$$
where $J_p$ denotes the Bessel function. Hence for the spaces
$$
\KK_{q,\mu}^{2}:= \set{\chi(t) = c\, \frac{\delta_{-\log \mu}(t)}{t^{2q-1}} + \frac b t + \phi(t) \ :\ c,b \in \C,\ \phi \in
L^2(m_q)}
$$
one finds
$$
\HH_{q,\mu}^{2} = \BB_q [\KK_{q,\mu}^{2}]
$$
and moreover
\begin{equation*}
\begin{array}{c}
\PP_{0,q} \left( \frac{c\, \mu^{\frac 1 x}}{x^{2q}} + \BB_{q}
\left[  \psi \right] \right) = \BB_{q} \left[ M\left(c\,
\frac{\delta_{-\log \mu}(t)}{t^{2q-1}} + \psi \right) \right] \\[0.3cm]
\PP_{1,q} \left( \frac{c\, \mu^{\frac 1 x}}{x^{2q}} + \BB_{q}
\left[  \psi \right] \right) = \BB_q \left[ N_q \left( c\,
\frac{\delta_{-\log \mu}(t)}{t^{2q-1}} + \psi \right) \right]
\end{array}
\end{equation*}
that is for any $\chi \in \KK_{q,\mu}^{2}$
$$
\PP_{0,q} (\BB_{q} \left[  \chi \right] ) = \BB_q [M(\chi)] \qquad
\qquad \PP_{1,q} (\BB_{q} \left[  \chi \right] ) = \BB_q
[N_q(\chi)]
$$
The spaces $\KK_{q,\mu}^{2}$ can be written as
$$
\KK_{q,\mu}^{2} = \mbox{Span}_\C \, \left( \frac{\delta_{-\log \mu}(t)}{t^{2q-1}} \right) \oplus \mbox{Span}_\C \, \left(
\frac{1}{t} \right) \oplus L^2(m_q)
$$
\end{remark}

\begin{remark} \label{da-2-a-1}
Notice that (\ref{p-s0-hq}) holds also on $\HH^{1}_{q,\mu}$ for $\PP_{0,q}$ but not for $\PP_{1,q}$. In fact $N_{q}(\phi)$ is not defined for all functions $\phi \in L^{1}(m_{q})$.
\end{remark}

\begin{theorem} \label{eigenfunc-acca}
If $f\in \HH(B)$ and $\PP_{q}^{\pm} f = \lambda f$ with $\lambda
\in \C \setminus [0,1)$ then $f(x) \in \HH_{q,\lambda}^{2}$.
\end{theorem}

\begin{proof}
We use the properties of the inverse Laplace transform. In
particular, we recall
\begin{lemma} \label{anti-laplace}
A function $u(x)$ is the Laplace transform of a generalised
function if and only if there exists $k\in \R$ such that $u(x)$ is
holomorphic in the half-plane $\set{\Re(x)>k}$ and
$$
|u(x)| \le M \, (1+|x|)^{m} \qquad \Re(x)>k
$$
for some constants $M,m$.
\end{lemma}
\noindent
Let now $f \in \HH(B)$ satisfy $\PP_{q}^{\pm} f = \lambda f$ for
some $\lambda \in \C \setminus [0,1)$. We study separately the
cases $\lambda =1$ and $\lambda \in \C \setminus [0,1]$.
\noindent
Let us first consider the case $\lambda =1$. By Proposition
\ref{proprieta-farey}.(ii), the function $f$ satisfies eq.
(\ref{lewis}), and we can write
\begin{equation} \label{expansion}
f(x)= \sum_{n=0}^{\infty}\ a_{n}\, (x-1)^{n} \qquad \text{for}\ \ \set{|x-1|<1}
\end{equation}
where the sequence $\set{a_{n}}$ satisfies $\limsup_{n}
(a_{n})^{\frac 1 n} \le 1$, and the convergence is uniform on any
compact set contained in $\set{|x-1|<1}$. By Lemma \ref{anti-laplace}, the
right-hand side of eq. (\ref{lewis}) is the Laplace transform
of a generalised function. By using (\ref{power-trans}) and
(\ref{expansion}) and known properties of the Laplace transform, we
obtain
\begin{equation} \label{trans-right}
\LL^{-1} \left[ \left( \frac{1}{x+1} \right)^{2q} \ f\left(
\frac{x}{x+1} \right) \right] (t)= t_+^{2q-1} \ e^{-t} \
\sum_{n=0}^\infty\ \frac{(-1)^n a_n}{\Gamma(n+2q)}\ t^n
\end{equation}
The left hand-side of eq. (\ref{lewis}) is then the Laplace
transform of a generalised function. Moreover, from eq.
(\ref{lewis}) with $\lambda=1$ we obtain
$$
f(x+n) - f(x) = \sum_{h=0}^{n-1}\, (f(x+h+1)-f(x+h)) =  -
\sum_{h=0}^{n-1}\, \frac{f\left( \frac{x+h}{x+h+1}
\right)}{(x+h+1)^{2q}}
$$
Hence $f$ satisfies the assumptions of Lemma \ref{anti-laplace}
for any $k>0$. Hence
\begin{equation} \label{trans-left}
\LL^{-1} [f(x)-f(x+1)] = (1-e^{-t}) \, \LL^{-1} [f]
\end{equation}
Putting together (\ref{trans-right}) and (\ref{trans-left}), we
obtain that (\ref{lewis}) implies that there exists a constant
$c\in \C$ such that
\begin{equation} \label{trans-all}
f(x) = c + \LL \left[ \frac{t_+^{2q-1} \ e^{-t}}{1 - e^{-t}} \
\sum_{n=0}^\infty\ \frac{(-1)^n a_n t^n}{\Gamma(n+2q)} \right](x)
\end{equation}
Moreover, since by Proposition \ref{proprieta-farey}.(ii) it holds
$\JJ_{q}f = \pm f$, we can apply the operator $\JJ_{q}$ to the
right-hand side of (\ref{trans-all}) to have $\pm f$. It is a
straightforward computation that applying $\JJ_{q}$ to the Laplace
transform gives the $\BB_{q}$ transform defined in
(\ref{borel-def}), hence
\begin{equation} \label{trans-finale}
f(x) = \pm \left( \frac{c}{x^{2q}} + \BB_{q} \left[
\frac{e^{-t}}{1-e^{-t}}\ \sum_{n=0}^\infty\ \frac{(-1)^n a_n
t^n}{\Gamma(n+2q)} \right](x) \right)
\end{equation}
and the thesis follows for $\lambda=1$, with $\mu=1$, $b=\pm
\frac{a_{0}}{\Gamma(2q)}$ and $\phi(t)$ as in (\ref{typical}).
\noindent
Let us now consider the case $\lambda \in \C \setminus [0,1]$. Let
us first assume that $f$ is bounded at $x=0$. In this case,
$\JJ_{q}f=\pm f$ implies that $|f(x)| = \OO(|x|^{-2\xi})$ as
$\Re(x)\to \infty$, hence $f$ satisfies the assumptions of Lemma
\ref{anti-laplace} with $k=0$. In particular, using again the
expansion (\ref{expansion}), we can write eq.
(\ref{trans-right}) and the analogous of (\ref{trans-left}) which
is
$$
\LL^{-1} [\lambda f(x)-f(x+1)] = (\lambda-e^{-t}) \, \LL^{-1}[f]
$$
Hence in this case we get
\begin{equation} \label{trans-all-not-1}
f(x) = \LL \left[ \frac{t_+^{2q-1} \ e^{-t}}{\lambda - e^{-t}} \
\sum_{n=0}^\infty\ \frac{(-1)^n a_n t^n}{\Gamma(n+2q)} \right](x)
\end{equation}
By using again the same argument as before applying the involution
$\JJ_{q}$ to the right-hand side of (\ref{trans-all-not-1}), it
follows that
\begin{equation} \label{trans-finale-2}
f(x) = \BB_{q} \left[  \frac{e^{-t}}{\lambda-e^{-t}}\,
\sum_{n=0}^\infty\ \frac{(-1)^n a_n t^{n}}{\Gamma(n+2q)} \right](x)
\end{equation}
hence $f\in \HH_{q,\lambda}^{2}$ with $c=b=0$, and
\begin{equation} \label{fi-not-1}
\phi(t) = \frac{e^{-t}}{\lambda-e^{-t}}\, \sum_{n=0}^\infty\
\frac{(-1)^n a_n t^{n}}{\Gamma(n+2q)}
\end{equation}
which can be written in the form (\ref{typical}) with $\phi(0) =
\frac{1}{\lambda-1}\, \frac{a_0}{\Gamma(2q)}$.
\noindent
We finish the proof by studying the case $\lambda \in \C \setminus
[0,1]$ when the function $f(x)$ is not bounded at $x=0$. Let us
consider the function $g(x):= \frac{f(x)}{\lambda^{x}}$. By
eq. (\ref{lewis}) we can write
$$
f(x+n)= \lambda^{n}\, f(x) - \sum_{h=0}^{n-1}\ \lambda^{(n-h-1)}\,
\frac{f\left( \frac{x+h}{x+h+1} \right)}{(x+h+1)^{2q}}
$$
which implies that $g$ satisfies
$$
g(x+n) = g(x) - \sum_{h=0}^{n-1}\ \lambda^{-(x+h+1)}\,
\frac{f\left( \frac{x+h}{x+h+1} \right)}{(x+h+1)^{2q}}
$$
Hence depending on whether $|\lambda|\le 1$ or $|\lambda|> 1$,
either $f$ or $g$ satisfy the assumptions of Lemma
\ref{anti-laplace} for any $k>0$. Since eq.
(\ref{trans-right}) still applies, if $|\lambda|\le 1$ we can
repeat the same argument as above and obtain
(\ref{trans-all-not-1}). If $|\lambda|> 1$, eq. (\ref{lewis})
implies that $g$ satisfies
\begin{equation} \label{lewis-g}
g(x)-g(x+1) = \frac{1}{\lambda^{x+1}} \left( \frac{1}{x+1}
\right)^{2q}  f\left( \frac{x}{x+1} \right) = \sum_{n=0}^{\infty}\
(-1)^{n} a_{n}\, \frac{e^{-(x+1)\log \lambda}}{(x+1)^{n+2q}}
\end{equation}
where we have again used the expansion (\ref{expansion}) for $f$.
Both the right-hand side of (\ref{lewis-g}) and $g$ satisfy the
assumptions of Lemma \ref{anti-laplace}. Hence we can apply
$\LL^{-1}$ to both sides of (\ref{lewis-g}) and write the
analogous of (\ref{trans-left}) for $g$, to obtain that there
exists $c\in \C$ such that
\begin{equation} \label{trans-all-g}
g(x) = c + \LL \left[ \frac{(t- \log \lambda)_+^{2q-1} \ e^{-t}}{1
- e^{-t}} \ \sum_{n=0}^\infty\ \frac{(-1)^n a_n }{\Gamma(n+2q)} \,
(t- \log \lambda)^{n} \right](x)
\end{equation}
where we have used the notation $(t- \log \lambda)_+ := (t-\log
|\lambda|)_+ - i \arg \lambda$, and the relation
$$
\LL^{-1}\left[ \frac{e^{-(x+1)\log \lambda}}{(x+1)^{n+2q}} \right](t)
= \frac{e^{-t} (t- \log \lambda)_+^{n+2q-1}}{\Gamma(n+2q)}
$$
Now, since $\JJ_{q}f = \pm f$, we obtain from (\ref{trans-all-g})
and the definition of (\ref{borel-def})
\begin{equation} \label{trans-finale-3}
f(x) = \pm \frac{\lambda^{\frac 1 x}}{x^{2q}}\, g\left( \frac 1 x
\right) = \pm \left( \frac{c\, \lambda^{\frac 1 x}}{x^{2q}} +
\BB_{q} \left[  \frac{e^{-t}}{\lambda-e^{-t}}\, \sum_{n=0}^\infty\
\frac{(-1)^n a_n t^{n}}{\Gamma(n+2q)} \right] \right)
\end{equation}
and the thesis follows with $b=0$ and $\phi$ as in
(\ref{fi-not-1}).
\end{proof}

\begin{corollary} \label{autofunz}
If $f\in \HH(B)$ and $\PP_{q}^{\pm} f = \lambda f$ with $\lambda
\in \C \setminus [0,1)$ then it has the form
\begin{equation} \label{forma-autofunz}
f(x) = \frac{c\, \lambda^{\frac 1 x}}{x^{2q}} +
\frac{\Gamma(2q-1)}{\Gamma(2q)}\, \frac b x + \BB_q[\phi]
\end{equation}
with $c,b \in \C$ and $\phi \in L^2(m_q)$ with $\phi(0)$ finite and $\phi(t) = \phi(0) + O(t)$ as $t\to 0^{+}$. Moreover if $\lambda
\not= 1$ then $b=0$, if $\lambda =1$ then $b = f(1)$. If $f$ is an eigenfunction of $\PP_{q}^{-}$ then $b=f(1)=0$. Finally the function $\phi \in L^2(m_q)$ is such that $\BB_q[\phi]$ is bounded as $x\to 0$.
\end{corollary}

\begin{proof}
The form (\ref{forma-autofunz}) follows from the proof of Theorem \ref{eigenfunc-acca} (see (\ref{trans-finale}),(\ref{trans-finale-2}),(\ref{trans-finale-3})) and (\ref{bq-su-t}). That $b=0$ if $\lambda \not= 1$ follows from (\ref{trans-finale-2}) and (\ref{trans-finale-3}), and that $b=f(1)$ if $\lambda =1$ follows from (\ref{trans-finale}). That $f(1)=0$ for eigenfunctions of $\PP_{q}^{-}$ for any $\lambda$ follows from Proposition \ref{proprieta-farey}-(ii).
\noindent
We now have to prove boundedness of $\BB_q[\phi]$ as $x\to 0$. We start with the case $\lambda=1$. From eq. (\ref{trans-finale}), it follows that we have to prove boundedness of $\BB_{q}[\psi]$ with
\begin{equation} \label{la-psi-limit}
\psi(t) := \frac{e^{-t}}{1-e^{-t}}\ \sum_{n=0}^\infty\ \frac{(-1)^n a_n
t^n}{\Gamma(n+2q)} - \frac{a_{0}}{\Gamma(2q)}\, \frac 1 t
\end{equation}
Since for any $\phi \in L^{2}(m_{q})$ the definition of the $\BB_{q}$ transform implies
$$
\BB_q[\phi] \left( \frac 1 x \right) = x^{2q}\ \LL \left[ t_{+}^{2q-1}\, \phi \right] (x)
$$
we need to show that $x^{2q}\ \LL \left[ t_{+}^{2q-1}\, \psi \right](x)$ is bounded as $x \to \infty$. Since $\limsup_{n} (a_{n})^{\frac 1 n} \le 1$, the power series in (\ref{la-psi-limit}) converges uniformly for $t\in \R$. Hence
\begin{equation} \label{scambio-1}
\LL \left[ t_{+}^{2q-1}\, \psi \right] = \sum_{n=1}^\infty\ \frac{(-1)^n a_n}{\Gamma(n+2q)} \ \LL \left[ \frac{ t_{+}^{n+2q-1}\, e^{-t}}{1-e^{-t}} \right] + \frac{a_{0}}{\Gamma(2q)}\, \LL \left[ \frac{ t_{+}^{2q-1}\, e^{-t}}{1-e^{-t}} - t_{+}^{2q-2} \right]
\end{equation}
Moreover, since $e^{-t} <1$ for $t>0$, for all $n\ge 1$ and $\xi>0$ we get
\begin{equation} \label{scambio-2}
\LL \left[ \frac{ t_{+}^{n+2q-1}\, e^{-t}}{1-e^{-t}} \right] = \sum_{k=1}^\infty\ \LL \left[ e^{-kt} \, t^{n+2q-1}_{+} \right] = \sum_{k=1}^\infty\ \frac{\Gamma(n+2q)}{(x+k)^{n+2q}} = \Gamma(n+2q) \ \zeta_{H}(n+2q, x+1)
\end{equation}
where $\zeta_{H}(s,a)$ denotes the Hurwitz zeta function, and we have used uniform convergence of the series of functions in the first equality. The same argument works for $n=0$ if $\xi> \frac 1 2$, hence using analytic continuation of $\zeta_{H}(2q,x+1)$ to $\xi> \frac 1 2,\, q \not= \frac 1 2$, we write
\begin{equation} \label{scambio-3}
\LL \left[ \frac{ t_{+}^{2q-1}\, e^{-t}}{1-e^{-t}} - t_{+}^{2q-2} \right] = \Gamma(2q) \ \zeta_{H}(2q, x+1) - \Gamma(2q-1)\ \frac{1}{x^{2q-1}}
\end{equation}
Putting together (\ref{scambio-1}), (\ref{scambio-2}) and (\ref{scambio-3}), we get
\begin{equation} \label{con-hurw}
x^{2q}\, \LL \left[ t_{+}^{2q-1}\, \psi \right](x) =x^{2q}\, \sum_{n=0}^\infty\ (-1)^n a_n\, \zeta_{H}(n+2q, x+1) - a_{0}\, \frac{\Gamma(2q-1)}{\Gamma(2q)}\, \frac{x^{2q}}{x^{2q-1}}
\end{equation}
To study the behaviour of (\ref{con-hurw}) as $x\to \infty$, we use the following formula (see \cite[pag.269]{A1}) valid for $\Re(n+2q)>0, n+2q \not= 1$
\begin{equation} \label{prol-hurw}
\zeta_{H}(n+2q,x+1) = \frac{1}{(x+1)^{n+2q}} + \frac{(x+1)^{1-n-2q}}{n+2q-1} - (n+2q)\, \int_{0}^{\infty}\ \frac{t- \lfloor t \rfloor}{(t+x+1)^{n+2q+1}}\, dt
\end{equation}
Using (\ref{prol-hurw}), we write (\ref{con-hurw}) as a sum of different terms. First, we isolate the $n=0$ term and the last term in (\ref{con-hurw}), which give
$$
a_{0}\, \frac{x^{2q}}{(x+1)^{2q}} + \frac{a_{0}}{2q-1} \frac{x^{2q}}{(x+1)^{2q}} \, (x+1) -
\frac{a_{0}}{2q-1}\, x - a_{0}\, 2q\, x^{2q}\ \int_{0}^{\infty}\ \frac{t- \lfloor t \rfloor}{(t+x+1)^{2q+1}}\, dt =
$$
$$
= O(1) + \frac{a_{0}}{2q-1} \left( \frac{x^{2q}}{(x+1)^{2q}} -1 \right) \, x + \frac{a_{0}}{2q-1} \frac{x^{2q}}{(x+1)^{2q}} + O(1) = O(1)
$$
as $x \to \infty$, since
$$
\left| x^{2q}\,  \int_{0}^{\infty}\ \frac{t- \lfloor t \rfloor}{(t+x+1)^{2q+1}}\, dt \right| \le x^{2\xi}\, \int_{0}^{\infty}\ \frac{1}{(t+x+1)^{2\xi+1}}\, dt = \frac{1}{2\xi}\,  \frac{x^{2\xi}}{(x+1)^{2\xi}}
$$
Now we come to the terms in (\ref{con-hurw}) with $n\ge 1$. From (\ref{prol-hurw}) we get
$$
\left| x^{2q}\ \sum_{n=1}^\infty\ (-1)^n a_n\, \zeta_{H}(n+2q, x+1) \right| \le
$$
$$
\le \left| \frac{x^{2q}}{(x+1)^{2q}}\right| \ \sum_{n=1}^\infty\ \frac{|a_n|}{(x+1)^{n}} + \left| \frac{x^{2q}}{(x+1)^{2q}}\right|\ \sum_{n=1}^\infty\ \frac{|a_n|}{|n+2q-1|\, (x+1)^{n-1}} +
$$
$$
+ |x^{2q}|\ \sum_{n=1}^\infty\ |a_{n}|\, |n+2q|\, \int_{0}^{\infty}\ \frac{1}{\left| (t+x+1)^{n+2q+1} \right|} \, dt
$$
and the right hand side is $O(1)$ as $x\to \infty$. Hence the thesis follows for $\lambda=1$.
\noindent
For $\lambda \not= 1$, boundedness of $\BB_{q}[\phi]$ follows from  (\ref{trans-finale-2}) and (\ref{trans-finale-3}). Indeed in eq. (\ref{trans-finale-2}), the function $f(x)$ is bounded as $x\to 0$ by assumption, and $f(x) = \BB_{q}[\phi]$. Moreover from eq. (\ref{trans-finale-3}) it follows that the function $\phi$ is the same as in (\ref{trans-finale-2}), hence $\BB_{q}[\phi]$ is again bounded as $x\to 0$.
\end{proof}
\noindent
An example of eigenfunctions of $\PP_q^+$ with $\lambda =1$ is the
family of functions defined as
$$
f_{q}^{+}(x) = \frac{\zeta_{R}(2q)}{2}\, \left(1 + \frac{1}{x^{2q}}
\right) + \sum_{m,n\ge 1} \, \frac{1}{(mx+n)^{2q}} \qquad \qquad
\mbox{for}\ \  \Re(q)>1
$$
where $\zeta_{R}(s)$ denotes the Riemann zeta-function. It is shown in
\cite{LeZa} that the family $f_{q}^{+}$ can be analytically
continued to $q\in \C$. In particular $f_{1}^{+}(x) = \frac 1 x$
is the invariant density of the Farey map and the only
eigenfunction of $\PP_{1}^{+}$ with $\lambda =1$. We can write the
functions $f_q^+$ as in (\ref{forma-autofunz}). First of all,
notice that
$$
\sum_{m,n\ge 1} \, \frac{1}{(mx+n)^{2q}} = \frac{1}{x^{2q}}\,
\sum_{m,n\ge 1} \, \frac{1}{n^{2q} \left( \frac m n + \frac 1 x
\right)^{2q}} = \frac{1}{\Gamma(2q)\, x^{2q}}\, \sum_{m,n\ge 1} \,
\frac{1}{n^{2q}}\, \LL \left[ t_+^{2q-1} e^{-\frac m n\, t}
\right] \left(\frac 1 x \right) =
$$
$$
= \frac{1}{\Gamma(2q)}\, \sum_{m,n\ge 1} \, \frac{1}{n^{2q}}\,
\BB_q \left[ e^{-\frac m n\, t} \right] = \BB_q \left[
\frac{1}{\Gamma(2q)}\, \sum_{m,n\ge 1} \, \frac{e^{-\frac m n\,
t}}{n^{2q}}\, \right] = \BB_q \left[ \frac{1}{\Gamma(2q)}\,
\sum_{n\ge 1} \, \frac{1}{n^{2q}}\, \frac{e^{-\frac t
n}}{1-e^{-\frac t n}} \right] =
$$
$$
= \BB_q \left[ \frac{1}{\Gamma(2q)\, t}\, \sum_{k\ge 0} \,
\frac{B_k}{k!}\, t^k \, \left( \sum_{n\ge 1} \,
\frac{1}{n^{k+2q-1}} \right) \right] = \BB_q \left[
\frac{1}{\Gamma(2q)\, t}\, \sum_{k\ge 0} \, \frac{B_k\,
\zeta_{R}(k+2q-1)}{k!}\, t^k \right]
$$
where $\set{B_k}$ are the Bernoulli numbers, and the series
converges for $|t|< 2\pi$ since $\zeta_{R}(k+2q-1) = O(1)$ as $k \to
\infty$. Hence we get
\begin{equation} \label{zagier-f}
f_{q}^{+}(x) = \frac{\zeta_{R}(2q)}{2}\, \frac{1}{x^{2q}} +
\frac{\Gamma(2q-1)}{\Gamma(2q)}\, \frac{\zeta_{R}(2q-1)}{x} +
\BB_q\left[ \phi \right]
\end{equation}
where $\phi$ is given for $|t|< 2\pi$ by
$$
\phi(t) = \frac{\zeta_{R}(2q)}{2\, \Gamma(2q)} +
\frac{1}{\Gamma(2q)}\, \sum_{k\ge 1} \, \frac{B_k\,
\zeta_{R}(k+2q-1)}{k!}\, t^{k-1}
$$
Moreover in \cite{LeZa} it is also proved that any eigenfunction
of $\PP_q^+$ with $\lambda =1$, when written as in
(\ref{forma-autofunz}), satisfies $c= \alpha\, \frac 1 2\,
\zeta_{R}(2q)$ and $b = \alpha\, \zeta_{R}(2q-1)$ for some $\alpha \in
\C$ (see \cite[Remark 1, pag.246]{LeZa}).
\noindent
An example of eigenfunctions of $\PP_q^-$ with $\lambda =1$ is the family of functions
$$
f_{q}^{-}(x) = 1- \frac{1}{x^{2q}} = - \frac{1}{x^{2q}} + \BB_q
\left[ \frac{1}{\Gamma(2q)} \right]
$$
in particular $b=0$ as stated in Corollary \ref{autofunz}.
Moreover, any other eigenfunction of $\PP_q^-$ with $\lambda =1$
can be written as in (\ref{forma-autofunz}) with $c=b=0$. Indeed,
$b=0$ since it is an eigenfunction of $\PP_q^-$ as stated in
Corollary \ref{autofunz}, and $c$ can be eliminated by subtracting
a multiple of $f_{q}^{-}$.

\section{Induced operators} \label{induced-op}
\noindent
The Farey map $F$ has at least two induced versions. The first
one is the well-known \emph{Gauss map} $G$ which is obtained by
iterating $F$ once plus the number of times necessary to reach the
interval $[1/2,1]$. The map $G:[0,1]\to [0,1]$ is given by
\begin{equation} \label{indg}
G(x) = \left\{
\begin{array}{cl}
F^{[1/x]}(x)=\set{ \frac{1}{x} } & \ x>0 \\[0.3cm]
0 & \ x=0
\end{array} \right.
\end{equation}
It is easily checked that the Gauss map acts on the continued fraction expansion of a number $x\in [0,1]$ as:
$x=[a_1,a_2,a_3,\dots]$ implies $G(x) =[a_2,a_3,\dots]$.
\noindent
We now introduce a family of operators related to the Gauss map. Recall that we assume $\xi=\Re(q)>0$. Consider the spaces
\begin{equation} \label{spazi-acca-meno}
\HH^{p}_q := \set{g \in \HH^{p}_{q,\mu}\ :\ c=b=0} \subset \HH(B), \qquad \text{$\xi>0$ and $p\in [1,+\infty]$}
\end{equation}
A function $g$ in $\HH^{p}_q$ is a $\BB_q$ transform of an $L^p(m_q)$ function. Moreover, let
\begin{equation} \label{spazi-acca-meno-tilde}
\tilde \HH^{p}_q := \set{g \in \HH^{p}_{q}\ :g = \BB_{q}[\phi] \ \text{with $\phi(t) = \phi(0) +O(t^{\eps})$ for some $\eps >0$ as $t\to 0^{+}$}} 
\end{equation}
Recall that from the definition of the $L^{p}(m_{q})$ spaces (see \eqref{spazi-lpq}), it follows that $\HH^{p}_{q} \subset \HH^{1}_{q}$ and $\tilde \HH^{p}_{q} \subset \tilde \HH^{1}_{q}$ for all $p\ge 1$. 

We introduce the family of operators
$$
\QQ_{q,z} : \HH^{1}_q \to \HH(B), \qquad z\in \C\setminus [1,+\infty)\
\mbox{ and }\ \xi>0
$$
$$
\QQ_{q,1} : \tilde \HH^{1}_q \to \HH(B), \qquad \xi> \frac 1 2
$$
defined by
\begin{equation} \label{q}
g(x)= \BB_q[\phi](x)\ \longmapsto\ (\QQ_{q,z} g)(x) = z\,
\PP_{1,q}\, \BB_q\left[ (1-z\, e^{-t})^{-1}\ \phi(t) \right](x)
\end{equation}
The operators are well-defined since $(1-z\,
e^{-t})^{-1}\ \phi(t)$ is in $L^1(m_q)$ for any $\phi$ in
$L^1(m_q)$, being $(1-z\, e^{-t})^{-1}$ in $L^{\infty}$ for $z\in \C\setminus [1,+\infty)$. If instead $z=1$ then $(1- e^{-t})^{-1}\ \phi(t)$ is in $L^1(m_q)$ if $\phi(0)$ is finite since $\frac 1 t$ is in $L^{1}(m_{q})$ for $\xi > \frac 1 2$. 

\begin{remark} \label{isom-l2}
Notice that $\QQ_{q,z}(\HH^{p}_{q})$ is contained in $\HH^{2}_{q}$ for all $p\ge 2$, since by Remark \ref{immagin-p1} the set $\PP_{1,q}(\HH^{2}_{q})$ is in $\HH^{2}_q$. Hence, using Proposition \ref{azione-su-spazi}, the
operators $\QQ_{q,z}$ induce the operators
$$
Q_{q,z} : L^2(m_q) \to L^2(m_q)
$$
given by
\begin{equation} \label{q-l2}
\phi\ \longmapsto\ Q_{q,z} \phi = z\, N_{q}\, (1-z\, M)^{-1}\ \phi
\end{equation}
Notice that for $z=1$, the condition $Q_{q,1}\phi = \phi$ is identical to the integral equation studied by Lewis in \cite{Le}.
\end{remark}

\begin{theorem} \label{thm:lerch-zeta}
The operators $\QQ_{q,z}$ admit the integral representation
\begin{equation} \label{int-rep}
(\QQ_{q,z}\, \BB_q[\phi(t)])(x)\ =\ z\ \int_0^\infty\,
\frac{e^{-t(x+1)}\, t^{2q-1}}{(1-z\, e^{-t})}\, \phi(t)\ dt \ \in \HH(B)
\end{equation}
In particular for $\xi>0$ the function $z
\mapsto \QQ_{q,z}g$ is analytic in $\C \setminus [1,+\infty)$ for any $g\in \HH^{1}_{q}$.
\end{theorem}

\begin{proof}
The integral representation (\ref{int-rep}) is a straightforward
consequence of definitions of $\BB_{q}$ in (\ref{borel-def}) and
$\PP_{1,q}$ in (\ref{p1}). Moreover, since 
$$
|e^{-t(x+1)}| \le  e^{-t} \qquad \forall\, x\in B
$$
and $(1-z\, e^{-t})^{-1}\ \phi(t)$ is in $L^1(m_q)$ as defined in \eqref{spazi-lpq}, the integral in (\ref{int-rep}) is finite. The analyticity of $z\mapsto \QQ_{q,z}g$ on $z \in \C \setminus [1,+\infty)$ for $\xi>0$ follows by the uniform convergence of the integral representation.
\end{proof}

\begin{remark} \label{rem:lerch}
Notice that for $\phi(t) = 1$, which is in $L^1(m_q)$ for $\xi>0$, we have
$$
(\QQ_{q,z}\, \BB_q[1])(x) = z\, \Gamma(2q)\, \Phi(z, 2q, x+1)
$$
where $\Phi(z,s,a)$ is the Lerch zeta function, defined for $|z|<1$ and $\Re(a)>0$ as
$$
\Phi(z,s,a) = \sum_{k=0}^{\infty}\ \frac{z^{k}}{(k+a)^{s}}
$$
(see e.g. \cite[vol.II, pag.27]{E}).
\end{remark}

\begin{theorem} \label{gauss-map}
If $g\in \tilde \HH^{1}_q$ then we can write
 \begin{equation} \label{ps}
(\QQ_{q,z} g)(x) = \sum_{n\ge 1}
\frac{z^n}{(x+n)^{2q}} \ g\left( \frac{1}{x+n} \right)
\end{equation}
for $|z|<1$ if $\xi>\frac 12$, and for $|z|\le 1$ if $\xi> 1$.
\end{theorem}

\begin{proof}
Using definition (\ref{q}), write formally
\begin{eqnarray}
(\QQ_{q,z} g)(x) &=& z \PP_{1,q} \BB_{q} \left[ \sum_{n=0}^{\infty}\ z^n\, e^{-nt}\, \phi(t) \right] \nonumber \\
&=&\sum_{n=0}^{\infty}\ z^{n+1}\, \int_0^\infty\ e^{-t(x+1)}\,
t^{2q-1}\, e^{-nt}\, \phi(t)\ dt  \nonumber \\
&=& \sum_{n=0}^{\infty}\ \frac{z^{n+1}}{(x+n+1)^{2q}}\, \BB_q \left[
\phi \right] \left( \frac{1}{x+n+1} \right)\nonumber \\
&=& \sum_{n\ge 1}
\frac{z^n}{(x+n)^{2q}} \ g\left( \frac{1}{x+n} \right) \nonumber
\end{eqnarray}
The conditions for convergence follow from conditions for uniform convergence of the series in the second equality.
\end{proof}

\begin{remark} \label{rapporto-gauss}
From (\ref{ps}) it follows that the operators $\QQ_{q,1}$ coincide with the generalised transfer operators of the Gauss map $G$ for functions $g\in \tilde \HH^{1}_q$.
\end{remark}
\noindent
We now study the relations between $\QQ_{q,z}$ and $\PP_q^\pm$.

\begin{theorem} \label{q-vs-p}
Let $f \in \HH^{1}_{q,\mu}$ with $z = \frac 1 \mu \in \C \setminus (1,\infty)$. Then for $\xi=\Re(q)> \frac 12$
\begin{equation} \label{ug-funz}
\left(1\mp \QQ_{q,z} \right)\, \left(1-z\, \PP_{0,q} \right)\, f\
=\ (1- z\, \PP_q^\pm )\, f \pm c\, \mu^x
\end{equation}
If $\xi \le \frac 12, \, q\not= \frac 12$, then equality (\ref{ug-funz}) holds for $f \in \HH^{1}_{q,1}$, and if $\mu\not= 1$ for functions $f\in \HH^{1}_{q,\mu}$ with $b=0$.
\end{theorem}

\begin{proof}
First of all, we show that the left-hand side of (\ref{ug-funz})
is well defined. This follows using Proposition
\ref{azione-su-spazi} and Remark \ref{da-2-a-1}, and writing
\begin{eqnarray}
\left(1-z\, \PP_{0,q} \right)\, f &=& \left(1-z\, \PP_{0,q}
\right)\, \left( \frac{c\, \mu^{\frac 1 x}}{x^{2q}} +
\BB_q\left[\frac b t + \phi(t)\right] \right) \nonumber \\
&=& \frac{c\, \mu^{\frac 1 x}}{x^{2q}} + \BB_q\left[\frac b t +
\phi(t)\right] - z \ \left( \frac{c\, \mu\, \mu^{\frac 1
x}}{x^{2q}} + \BB_q\left[ M \left(\frac b t + \phi(t)\right)
\right] \right) \nonumber \\
&=& \BB_q\left[ \frac b t \left( 1 - z \right) \right] + \BB_q
\left[ \phi(t)\, \left( 1 - z\, e^{-t} \right) + z\,  \frac{b\,
(1-e^{-t})}{t} \right] \nonumber
\end{eqnarray}
The last term is in $\HH^{1}_q$ since the function in square brackets
is in $L^1(m_q)$ for any $\phi \in L^{1}(m_{q})$ and $\xi>0$. The first term is in $L^{1}(m_{q})$ for $\xi > \frac 1 2$, and vanishes for $\xi \le \frac 12$ since if $\mu \not=1 $ we have $b=0$.
\noindent
Hence we can now apply $(1- \QQ_{q,z})$. Assume first $\mu\not= 1$ and $\xi \le \frac 12$, then it follows that
\begin{eqnarray}
\left(1\mp \QQ_{q,z} \right)\, \left(1-z\, \PP_{0,q} \right)\, f\
&=& \left(1\mp \QQ_{q,z} \right)\, \BB_q \left[ \phi(t)\, \left( 1 -
z\, e^{-t} \right) \right] \nonumber \\
&=& \BB_q \left[ \phi(t)\, \left( 1 - z\, e^{-t} \right) \right] - z\, \PP_{1,q}\,  \BB_q \left[  \left( 1 - z\, e^{-t} \right)^{-1} \left( \phi(t)\, (1-z\, e^{-t}) \right) \right] \nonumber \\
&=& (1-z\, \PP_{0,q} - z\, \PP_{1,q})\, \BB_{q}[\phi] = (1- z\, \PP_q^+ )\, f \nonumber
\end{eqnarray}
where we have used again Proposition \ref{azione-su-spazi} and Remark \ref{da-2-a-1}.
\noindent
The case $\mu\not= 1$ and $\xi > \frac 12$ follows in the same way, by writing $f = \BB_{q}[\tilde \phi]$ with $\tilde \phi = \frac b t + \phi \in L^{1}(m_{q})$. In the case $\mu =1$ instead we get
\begin{eqnarray}
\left(1\mp \QQ_{q,1} \right)\, \left(1- \PP_{0,q} \right)\, f\
&=& \left(1\mp \QQ_{q,1} \right)\, \BB_q \left[ \phi(t)\, \left( 1 - e^{-t} \right)  +  \frac{b\, (1-e^{-t})}{t} \right] \nonumber \\
&=& \BB_q \left[ \phi(t)\, \left( 1 - e^{-t} \right) +  \frac{b\, (1-e^{-t})}{t} \right] \nonumber \\
&& \qquad \qquad-  \PP_{1,q}\,  \BB_q \left[  \left( 1 - e^{-t} \right)^{-1} \left( \phi(t)\, (1-z\, e^{-t}) +  \frac{b\, (1-e^{-t})}{t}  \right) \right]\nonumber \\
&=& (1-\PP_{0,q} - \PP_{1,q})\, \BB_{q}\left[ \frac b t+\phi \right] = (1- \PP_q^+ )\, f \nonumber
\end{eqnarray}
where we have used again Proposition \ref{azione-su-spazi} and Remark \ref{da-2-a-1}.
\end{proof}

\begin{corollary} \label{relaz-autof}
Let $z \in \C \setminus (1,\infty)$. The operator $\QQ_{q,z}$ has
an eigenfunction $g\in \HH^{1}_{q}$ with eigenvalue $\lambda_Q = \pm 1$ if and
only if $\PP_q^\pm$ has an eigenfunction $f \in \HH^{1}_{q,\frac 1 z}$ with eigenvalue
$\lambda_P = \frac 1 z$ and term $c=0$. Moreover the eigenfunctions $g$ and $f$ satisfy
\begin{equation} \label{g-vs-f-explicit}
g = f - z\, \PP_{0,q} f 
\end{equation}
\end{corollary}

\begin{proof}
If $f$ is in $\HH^{1}_{q,\lambda}$ satisfies $\PP_q^+\, f = \lambda\, f$,
with $\lambda \not= 1$ then we can apply Theorem \ref{q-vs-p}, since $b=0$
by Corollary \ref{autofunz}. Then by (\ref{ug-funz})
$(1- \frac 1 \lambda\, \PP_{0,q})\, f$ is an eigenfunction of $\QQ_{q,
\frac 1 \lambda}$ with eigenvalue $\lambda_Q=1$ if $c=0$. The same
follows in the case $\lambda =1$.

On the contrary, if $g = \BB_q [\phi] \in \HH^{1}_q$ satisfies
$\QQ_{q,z}\, g = g$ for $z\not= 1$, then the function
$$
f:= \BB_q [(1-z\, e^{-t})^{-1} \phi] \in \HH^{1}_{q}
$$
satisfies by (\ref{q})
$$
z\, \PP_{1,q}\, f = g = (1- z\, \PP_{0,q})\, f
$$
hence it is an eigenfunction of $\PP_q^+$ with eigenvalue
$\lambda_P = \frac 1 z$ and $c=0$. If $z=1$ we can repeat the same argument by using the fact that $\QQ_{q,1}$ is defined on $\tilde \HH^{1}_{q}$, hence $g = \BB_q [\phi]$ with $\phi(t) = \phi(0) + O(t^{\eps})$ for some $\eps>0$ as $t\to 0^{+}$. Hence we can write $\frac{\phi(t)}{(1-e^{-t})} = \frac{b}{t} + \bar \phi(t)$, with $b=\phi(0)$ and $\bar \phi \in L^{1}(m_{q})$ for all $\xi>0$, and $f\in \HH^{1}_{q,1}$ with $c=0$.
\end{proof}

\begin{corollary} \label{autofunz-gauss}
Let $z \in \C \setminus (1,\infty)$. The eigenfunctions $g=\BB_{q}[\phi]$ of
$\QQ_{q,z}$ with eigenvalue $\lambda_Q = \pm 1$ are in $\tilde \HH^2_q$ and are bounded at $x=0$ with $g(0) = \Gamma(2q)\, \phi(0)$ if $z=1$, whereas $g(0) = (1-z)\, \Gamma(2q)\, \phi(0)$ if $z\not= 1$.
\end{corollary}

\begin{proof}
Let $z\not= 1$. By Corollary  \ref{relaz-autof}, from any $g= \BB_q[\phi] \in \HH^1_q$ an eigenfunction of $\QQ_{q,z}$ with eigenvalue $\lambda_Q=\pm 1$, we obtain an eigenfunction $f\in \HH^1_q$ of $\PP^\pm_q$ with eigenvalue $\lambda_P = \frac 1 z$ and such that $c=0$ when written as in (\ref{forma-autofunz}), and (\ref{g-vs-f-explicit}) holds. Moreover, by Corollary \ref{autofunz}, we know that $f\in \HH^2_{q,\frac 1 z}$ and $b=0$, hence actually $f\in \HH^2_q$. Since $f = \BB_q [(1-z\, e^{-t})^{-1} \phi]$ and $(1-z\, e^{-t})^{-1} \in L^\infty$, then $\phi \in L^2(m_q)$ and $g\in \HH^2_q$. If $z=1$, we can repeat the same argument with few changes. In this case $f = \BB_q [(1- e^{-t})^{-1} \phi]$ and
$$
\frac{\phi(t)}{(1-e^{-t})} = \frac{\phi(0)}{t} + \bar \phi(t) \ \ \text{with} \ \ \bar \phi (t) = \frac{\phi(t)-\phi(0)}{1-e^{-t}} + \frac{\phi(0)}{1-e^{-t}} - \frac{\phi(0)}{t}
$$
From Corollary \ref{autofunz} it follows that $\bar \phi(t) \in L^2(m_q)$, hence $\phi \in L^2(m_q)$ and $g\in \HH^2_q$. Moreover, from (\ref{lewis}) and (\ref{g-vs-f-explicit}) it follows that $g(x) = z f(x+1)$, hence $g(0) = z f(1)$.
Finally Theorem \ref{eigenfunc-acca} shows that the eigenfunction $f$ is written as in (\ref{trans-finale}), (\ref{trans-finale-2}), (\ref{trans-finale-3}). Let $z=1$, using (\ref{g-vs-f-explicit}) we write $g= \BB_{q}[\phi]$ with
$$
\phi(t) = e^{-t}\, \sum_{n=0}^\infty\ \frac{(-1)^n a_n t^n}{\Gamma(n+2q)}
$$
with $\limsup_{n} \sqrt[n]{|a_{n}|} \le 1$. Since for $n\ge 0$
$$
\BB_q \left[e^{-t} t^{n} \right] = \frac{1}{x^{2q}}\
\int_{0}^{\infty} \ e^{- t \left(\frac 1 x +1\right)} t^{n+2q-1}  \, dt
= \frac{1}{x^{2q}}\ \LL[t^{n+2q-1}]\left(\frac 1 x +1\right) = \Gamma(n+2q)\,
\frac{1}{x^{2q}}\, \left( \frac{x}{x+1} \right)^{n+2q}
$$
it follows that 
\begin{equation} \label{serie-g}
g(x) = \frac{1}{(x+1)^{2q}}\, \sum_{n=0}^\infty\ (-1)^n a_n \left( \frac{x}{x+1} \right)^{n}
\end{equation}
which implies in particular $g(0) = \Gamma(2q) \phi(0)$. The same argument can be used for $z\not= 1$.
\end{proof}

\vskip 0.3cm
\noindent
We finish this section by considering a second induced map from (\ref{farey}). It is the \emph{Fibonacci map} $H$ which is defined by iterating $F$ once plus the number of times necessary to reach the interval $[0,1/2]$. The map $H$ is defined by
\begin{equation} \label{fibonacci-map}
H(x) = \left\{
\begin{array}{ll}
\frac{S_{_{2n+1}}\, x - S_{_{2n}}}{S_{_{2n+1}} - S_{_{2n+2}}\, x}
& \mbox{if }\ x \in \left[ \frac{S_{_{2n}}}{S_{_{2n+1}}}\, ,\,
\frac{S_{_{2n+2}}}{S_{_{2n+3}}}
\right)\\[0.5cm]
\frac{S_{_{2n+1}} - S_{_{2n+2}}\, x}{S_{_{2n+5}}\, x -
S_{_{2n+4}}} & \mbox{if }\ x \in \left(
\frac{S_{_{2n+3}}}{S_{_{2n+4}}}\, ,\,
\frac{S_{_{2n+1}}}{S_{_{2n+2}}} \right]
\end{array} \right.
\end{equation}
where $\set{S_{n}}_{n\ge 0}$ are the Fibonacci numbers with $S_{_{0}}=0$, $S_{_{1}}=1$. The corresponding operators are obtained by exchanging the roles of $\PP_{0,q}$ and $\PP_{1,q}$ in (\ref{q}). We recall from \cite{BGI} that the operators $N_{q}$ on $L^{2}(m_{q})$ are of trace class with spectrum
\begin{equation} \label{form-spettro-N}
\sigma(N_{q}) = \set{0} \cup \set{ (-1)^{k}\, \alpha^{2(q+k)}}_{k\ge 0}
\end{equation}
where $\alpha = \frac{\sqrt{5}-1}{2}$ and each eigenvalue is simple. Hence we introduce on $\HH^1_{q}$ the second family of operators
$$
\RR_{q,z} : \HH^1_{q} \to \HH(B),  \qquad z\in \C \setminus \set{ (-1)^{k}\, \alpha^{-2(q+k)}}_{k\ge 0} \mbox{ and }\ \xi>0
$$
defined by
\begin{equation} \label{r}
g(x)= \BB_q[\phi(t)]\ \longmapsto\ (\RR_{q,z} g)(x) = z\,
\PP_{0,q}\, \BB_q\left[ (1-z\, N_{q})^{-1}\ \phi(t) \right]
\end{equation}
By (\ref{form-spettro-N}) the operators are well defined and we get

\begin{theorem} \label{r-anal-series}
For $\xi>0$, the operator-valued function $z\mapsto \RR_{q,z}$ is meromorphic in $\C$ with simple poles at $\set{ (-1)^{k}\, \alpha^{-2(q+k)}}_{k\ge 0}$. Moreover for all $g\in \HH^1_{q}$ one has
\begin{equation} \label{ps-r}
(\RR_{q,z} g)(x)= \sum_{n\ge 1}
\frac{z^n}{(S_{_{n+1}}\, x+S_{_n})^{2q}} \ g \left( \frac{S_{_n}
\, x + S_{_{n-1}}}{S_{_{n+1}}\, x+S_{_n}} \right)
\end{equation}
which, by the growth property of the Fibonacci numbers, is
absolutely convergent for $|z|< \alpha^{-2\xi}$.
\end{theorem}

\begin{proof}
The first part follows from the definition of $\RR_{q,z}$ in (\ref{r}) and (\ref{form-spettro-N}). Eq. (\ref{ps-r}) follows instead by writing
$$
(\RR_{q,z} g)(x) = (z \, \PP_{0,q}\, (1-z\, \PP_{1,q})^{-1} g) (x) = \sum_{n\ge 1}\, z^{n} (\PP_{0,q} \, \PP_{1,q}^{n-1} g)(x)
$$
and using
$$
(\PP_{1,q}^{n-1} g)(x) = \frac{1}{(S_{_{n-1}}\, x+S_{_n})^{2q}} \ g \left( \frac{S_{_{n-2}}
\, x + S_{_{n-1}}}{S_{_{n-1}}\, x+S_{_n}} \right)
$$
which can be proved by induction.
\end{proof}

\begin{remark} \label{rapporto-fibo}
From (\ref{ps-r}) it follows that the operators $\RR_{q,1}$ coincide with the generalised transfer operators of the Fibonacci map $H$ for functions $g \in \HH(B)$.
\end{remark}
\noindent
Analogously to $\QQ_{q,z}$ we now study the relations between $\RR_{q,z}$ and $\PP_{q}^{\pm}$.

\begin{theorem} \label{r-vs-p}
Let $f \in \HH^2_{q,\mu}$, with $c=b=0$. Then for $z = \frac 1 \mu \in \C \setminus \set{ (-1)^{k}\, \alpha^{-2(q+k)}}$
\begin{equation} \label{ug-funz-r}
\left(1\mp \RR_{q,z} \right)\, \left(1-z\, \PP_{1,q} \right)\, f\
=\ (1- z\, \PP_q^\pm )\, f 
\end{equation}
\end{theorem}

\begin{proof}
First of all, that the left-hand side of (\ref{ug-funz-r})
is well defined follows easily from definitions. Notice that in this case we need $c=b=0$ to be sure that $(1-z\, \PP_{1,q})\, f$ is in $\HH^1_q$. 
\noindent
Hence let $f=\BB_{q}[\phi]$ with $\phi \in L^2(m_q)$. Applying $(1- \RR_{q,z})$. It follows
$$
\left(1\mp \RR_{q,z} \right)\, \left(1-z\, \PP_{1,q} \right)\, f\
= \left(1\mp \RR_{q,z} \right)\, \BB_q \left[ (1-z\, N_{q}) \phi(t) \right] =
$$
$$
= \BB_q \left[ (1-z\, N_{q} -z\, M)\phi(t) \right] = (1- z\, \PP_q^\pm )\, \BB_{q}[\phi].
$$
\end{proof}
\noindent
We remark that using the power series expansion (\ref{ps-r}), one can prove that relation (\ref{ug-funz-r}) holds for all $f\in \HH(B)$ for $|z|< \alpha^{-2\xi}$.

\begin{corollary} \label{relaz-autof-r}
Let $z \in \C \setminus \set{ (-1)^{k}\, \alpha^{-2(q+k)}}$. The operator $\RR_{q,z}$ has an eigenfunction $g\in \HH^2_q$ with eigenvalue $\lambda_R = \pm 1$ if and only if $\PP_q^\pm$ has an eigenfunction $f$ with eigenvalue
$\lambda_P = \frac 1 z$ and $f \in \HH_{q,\mu}$ with $c=b=0$.
\end{corollary}

\begin{proof}
The proof is as in Corollary \ref{relaz-autof}.
\end{proof}
\noindent
Putting together Corollaries \ref{relaz-autof} and \ref{relaz-autof-r}, it follows that

\begin{corollary} \label{relaz-finale}
Let $z\in \C \setminus \left( (1,\infty) \cup \set{(-1)^{k}\, \alpha^{-2(q+k)}}\right)$. Then $f\in \HH^2_{q,\mu}$ with $c=b=0$ is an eigenfunction of $\PP_q^\pm$ with eigenvalue $\lambda_P = \frac 1 z=\mu$ if and only if 
$$
f(x) = h_{0}(x) + h_{1}(x)
$$
with $h_{0}$ and $h_{1}$ eigenfunctions of $\QQ_{q,z}$ and $\RR_{q,z}$ respectively in $\HH^2_{q}$, with eigenvalues $\lambda_{Q}=\lambda_{R}= \pm 1$.
\end{corollary}

\begin{proof}
Simply use that $h_{0}:=(1 - z \PP_{0,q}) f$ and $h_{1}:=(1 - z \PP_{1,q}) f$.
\end{proof}

\begin{remark}
Recall that relation (\ref{ug-funz-r}) can be extended to functions in $\HH(B)$, and that Corollary \ref{relaz-autof} holds for functions $f\in \HH^2_{q,1}$ with $b$ not necessarily vanishing. An example of an eigenfunction as in Corollary \ref{relaz-finale} is given for $q=1$ by 
$$
\frac 1 x = \frac{1}{1+x} + \frac{1}{x(x+1)}
$$
\end{remark}

\section{Two-variable zeta functions of Ruelle and Selberg} \label{sect-zeta}
The first two-variable zeta function we are interested in can be written in terms of the Gauss map as
\begin{equation} \label{zeta-sel-gen}
Z(q,z):=\exp \left( -\sum_{n\ge 1} \ \frac{z^n}{n}\, \sum_{x=G^{2n}(x)} {\left| (G^{2n})'(x) \right|^{-q} \over 1- \left| (G^{2n})'(x) \right|^{-1}}
\right)
\end{equation}
For $z=1$ and $\xi=\Re(q) > 1/2$ the function $Z(q,1)$ coincides with the \emph{Selberg zeta function} for the full modular group. This follows from the well known
one-to-one correspondence between the length spectrum (with multiplicities) of the modular surface $PSL(2, \Z)\setminus \hp$ and the set of values $\log
|(G^{2n})'(x)|$ (see e.g. \cite{Se}). 
Another zeta function which naturally comes into the play \cite{Rue} is 
the \emph{Ruelle zeta function} of the Farey map $F$, defined for $|z|$ small enough by
\begin{equation} \label{zeta}
\zeta (q,z) := \exp \left( \sum_{n\ge 1} \ \frac{z^n}{n}\,
\sum_{x=F^n(x)} \left| (F^n)'(x) \right|^{-q} \right)
\end{equation}
We shall study these functions by means of the operator-valued functions dealt with in the previous section. Our approach is similar in
spirit to that used in \cite{Ma3} for $Z(q,1)$. But first we describe the correspondence between the periodic points of the map $F$ and those of its induced versions $G$ (\ref{indg}) and $H$ (\ref{fibonacci-map}). Denoting
${\rm Per}\, F$, ${\rm Per}\, G$ and
${\rm Per}\, H$ the corresponding subsets of $[0,1]$ we have 
$$
{\rm Per}\, F \setminus \{0\}\cup \{\alpha\} ={\rm Per}\, G  \setminus \{\alpha\}=  {\rm Per}\, H \setminus \{0\}
$$ 
From the definitions we immediately see that whenever $x$ belongs to either of these sets its continued fraction expansion has to be periodic, which we write
$$
x= [{\overline {a_1,\cdots, a_n}}]
$$
Denoting $p_F(x)$, $p_G(x)$ and $p_H(x)$ the corresponding periods we have
$$
p_F(x) = \sum_{i=1}^n a_i   \quad , \quad p_G(x)=n \quad , \quad p_H(x)= p_F(x)-\#\{i\in [1,n]\, : \, a_i=1\}
$$
In other words, if for a given map $T:[0,1]\to [0,1]$ we define the {\sl partition function}
$$
Z_n(q,T) :=\sum_{x=T^n(x)} \left| (T^n)'(x) \right|^{-q}
$$
then
$$
Z_n(q,F)=1+\sum_{m=1}^n {n\over m}\, Z_m(q,G)=\alpha^{2qn}+\sum_{m=1}^n {n\over m}\, Z_m(q,H)
$$
Let moreover
\begin{equation} \label{pressure}
\lambda (q) := \lim_{n\to \infty} {1\over n} \log Z_n(q,F)
\end{equation}
It follows from thermodynamic formalism that the above limit exists for all $q\in \R$ and is a differentiable and monotonically decreasing function for $q\in (-\infty, 1)$, with 
$\lim_{q\to 1_-}\lambda (q)=1$ and $\lambda (q)=1$ for all $q\geq 1$ (\cite{PS}). In particular, for $q\in (-\infty,1)$ the function $\zeta (q,z)$ converges absolutely for $|z|< 1/\lambda (q)$ and has a simple pole at $z=1/\lambda
(q)$.

\noindent
Our aim is now to express both $Z(q,z)$ and $\zeta(q,z)$ in terms of Fredholm determinants of the operators $\QQ_{q,z}$ introduced in Section \ref{induced-op}. Following Mayer \cite{Ma1}, we restrict the operators $\QQ_{q,z}$ to the Banach space $A_\infty (D)$ of functions which are holomorphic on $D$ and continuous on $\overline{D}$ where
$$
D:= \set{x\in \C\ :\ \left| x-1 \right| < \frac 3 2 - \eps}.
$$
for small $\eps>0$. The space $A_\infty(D)$ is the set of holomorphic functions on which it is natural to study the spectral properties of $\QQ_{q,z}$ written as in Theorem \ref{gauss-map}.

We now prove that

\begin{proposition} \label{immersione}
If $g$ is in $A_\infty(D)$ then $g$ is in $\tilde \HH^{2}_{q}$ and $g=\BB_q[\phi_{q}]$ with $\phi_{q}(t) = \phi_{q}(0) + O(t)$ as $t\to 0^+$.
\end{proposition}
\noindent
In particular since $A_\infty(D) \subset \tilde\HH^{2}_{q}$ we can study the action of $\QQ_{q,z}$ on $A_\infty(D)$ for all $z\in \C \setminus (1,+\infty)$ and $\xi>0$.

\begin{proof}[Proof of Proposition \ref{immersione}]
We recall the definition of the generalized Laguerre polynomials
\begin{equation} \label{lag-pol}
L_n^{2q-1} (t) = \sum_{m=0}^n\, \frac{\Gamma(n+2q)}{\Gamma(m+2q)\, (n-m)!}\, \frac{(-t)^m}{m!},\qquad \xi= \Re(q)>0
\end{equation}
which satisfy
\begin{equation} \label{trasf-en}
\BB_q[L_n^{2q-1}](x) = \frac{\Gamma(n+2q)}{n!}\, (-1)^n (x-1)^n
\end{equation}
Let now $g(x)$ be in $A_\infty(D)$. It can be expressed as a power series
$$
g(x) = \sum_{n=o}^\infty\, a_n\, (x-1)^n \qquad \ a_n \in \C 
$$
with
\begin{equation}\label{vel-an}
\limsup \, \sqrt[n]{|a_n|} \le \left(\frac 32 - \eps\right)^{-1} .
\end{equation}
Hence from (\ref{trasf-en}) we can write
$$
g(x) = \sum_{n=0}^\infty\, \frac{(-1)^n\, a_n\, n!}{\Gamma(n+2q)}\, \BB_q[L_n^{2q-1}](x) 
$$
Letting
$$
\phi_{q}(t) := \sum_{n=0}^\infty\, \frac{(-1)^n\, a_n\, n!}{\Gamma(n+2q)}\, L_n^{2q-1} (t)
$$
we need to prove that $\phi_{q} \in L^{2}(m_{q})$ and that $g=\BB_{q}[\phi_{q}]$.

That $\phi_{q} \in L^{2}(m_{q})$ follows from the following estimates on $\|L_{n}^{2q-1}(t)\|_{L^{2}(m_{q})}$. Using computations from \cite{lag}, we have
\begin{equation} \label{norma-l2-lag}
\begin{array}{c}
\|L_{n}^{2q-1}(t)\|_{L^{2}(m_{q})}^{2} = \int_{0}^{\infty}\, L_{n}^{2q-1}(t)\, L_{n}^{2\bar q-1}(t)\, t^{2\xi -1} e^{-t}\, dt = \\[0.3cm]
= \left( \begin{array}{c} n+2\imath \eta -1 \\ n \end{array}\right)\, \left( \begin{array}{c} n-2\imath \eta -1 \\ n \end{array}\right)\, \Gamma(2\xi)\, {}_{3}F_{2}\left( \begin{array}{c} -n, -n, 2\xi\\ -2\imath \eta -n +1, 2\imath \eta -n +1\end{array}; 1\right)
\end{array}
\end{equation}
where we set $\eta := \Im(q)$ and ${}_{3}F_{2}$ is the generalised hypergeometric function. Moreover
$$
{}_{3}F_{2}\left( \begin{array}{c} -n, -n, 2\xi\\ -2\imath \eta -n +1, 2\imath \eta -n +1\end{array}; 1\right) = \sum_{k=0}^{n}\, \frac{(-n)_{k}^{2}\, \Gamma(2\xi+k)\, \Gamma(-2\imath \eta -n +1)\, \Gamma(2\imath \eta -n +1)}{\Gamma(2\xi)\, \Gamma(-2\imath \eta -n +1+k)\, \Gamma(2\imath \eta -n +1+k)\, k!}
$$
where $(-n)_{k} = n(n-1)\dots (n-k+1)$ is the Pochhammer symbol, and
$$
\frac{\Gamma(-2\imath \eta -n +1)\, \Gamma(2\imath \eta -n +1)}{\Gamma(-2\imath \eta -n +1+k)\, \Gamma(2\imath \eta -n +1+k)} = \prod_{j=1}^{k}\, \frac{1}{(n-j)^{2}+4\eta^{2}}
$$
It follows that for $\xi>0$ and $\eta \not= 0$
$$\left| {}_{3}F_{2}\left( \begin{array}{c} -n, -n, 2\xi\\ -2\imath \eta -n +1, 2\imath \eta -n +1\end{array}; 1\right) \right| \le \frac{n^{2}}{\min\{ 1, 4\eta^{2}\}} \, \sum_{k=0}^{n}\, \left( \begin{array}{c} k+2\xi -1 \\ k \end{array}\right)\le \frac{(n+2\xi)^{2\xi +2}}{\min\{ 1, 4\eta^{2}\}}$$
Hence
\begin{equation} \label{finita}
\sum_{n=0}^{\infty}\, \left\| \frac{(-1)^n\, a_n\, n!}{\Gamma(n+2q)}\, L_n^{2q-1} (t) \right\|_{L^{2}(m_{q})} \le \sum_{n=0}^{\infty}\, \frac{|a_{n}|\, n!}{|\Gamma(n+2q)|}\, \frac{|\Gamma(n+2\imath \eta)|}{|\Gamma(2\imath \eta)|\, n!}\, \frac{\sqrt{\Gamma(2\xi)}\, (n+2\xi)^{\xi +1}}{\sqrt{\min\{ 1, 4\eta^{2}\}}}<\infty
\end{equation}
by \eqref{vel-an} and since
$$\frac{|\Gamma(n+2\imath \eta)|}{|\Gamma(n+2q)|} \le n\qquad \forall\, n\ge 1$$
by standard estimates. If instead $\xi>0$ and $\eta=0$ then \eqref{norma-l2-lag} reads
$$\|L_{n}^{2q-1}(t)\|_{L^{2}(m_{q})}^{2} = \frac{\Gamma(n+2q)}{n!}$$ and \eqref{finita} follows again. Condition \eqref{finita} implies that $\phi_{q}\in L^{2}(m_{q})$. Moreover it satisfies
$$
\phi_{q}(0) = \sum_{n=0}^\infty\, \frac{(-1)^n\, a_n\, n!}{\Gamma(n+2q)}\, L_n^{2q-1}(0) = \sum_{n=0}^\infty\, \frac{(-1)^n\, a_n\, n!}{\Gamma(n+2q)}\, \frac{\Gamma(n+2q)}{\Gamma(2q)\, n!} \in \C
$$
$$
\lim\limits_{t\to 0^+} \frac{\phi(t) - \phi(0)}{t} =  - \sum_{n=0}^\infty\, \frac{(-1)^n\, a_n\, n!}{\Gamma(n+2q)}\, \frac{\Gamma(n+2q)}{\Gamma(2q+1)\, (n-1)!} \in \C
$$
To finish the proof we have to show that $g(x)= \BB_q[\phi_{q}](x)$ for $x\in B$. This follows from
$$
\left| \BB_q[\phi_{q}] - \sum_{n=0}^N\, \frac{(-1)^n\, a_n\, n!}{\Gamma(n+2q)}\, \BB_q[L_n^{2q-1}] \right| \le
\int_0^\infty\, \left| e^{-\frac tx}\, t^{2q-1} \,  \sum_{n>N}\, \frac{(-1)^n\, a_n\, n!}{\Gamma(n+2q)}\, L_n^{2q-1}(t) \right|\, dt \le
$$
$$
\le \int_0^\infty\, \left| e^{-t}\, t^{2q-1} \,  \sum_{n>N}\, \frac{(-1)^n\, a_n\, n!}{\Gamma(n+2q)}\, L_n^{2q-1}(t) \right|\, dt =
\int_0^\infty\, \left| \sum_{n>N}\, \frac{(-1)^n\, a_n\, n!}{\Gamma(n+2q)}\, L_n^{2q-1}(t) \right|\, m_q(dt)
$$
and the last term vanishes as $N\to \infty$ since $\phi_{q} \in L^2(m_q) \subset L^1(m_q)$.
\end{proof}

We now recall that by (\ref{serie-g}) in the proof of Corollary \ref{autofunz-gauss}, the eigenfunctions $g\in \tilde \HH^{2}_{q}$ of $\QQ_{q,z}$ are in $A_{\infty}(D)$. Hence we get
\begin{theorem} \label{isomorfismo-analitic}
The operator-valued function $q \mapsto \QQ_{q,z}:A_{\infty}(D) \to \HH(B)$ is analytic in $\Re(q)>0$ for $z\in \C \setminus [1,\infty)$, and in $\Re(q)>\frac 12$ for $z=1$. Moreover for $z=1$, the function $q \mapsto \QQ_{q,1}$ can be extended to a meromorphic function in $\Re(q)>0$ with a simple pole at $q = \frac 1 2$ with residue the operator $g \mapsto (\RR_{\frac 1 2} g)(x) = g(0)$.
\end{theorem}

\begin{proof}
We use the integral representation of Theorem \ref{thm:lerch-zeta} and Proposition \ref{immersione} to write for each $g\in A_{\infty}(D)$
\begin{equation} \label{int-rep-fin}
(\QQ_{q,z}g)(x) = z\, \int_{0}^{\infty}\, \frac{e^{-t(x+1)}\, t^{2q-1}}{1-ze^{-t}}\, \sum_{n=0}^\infty\, \frac{(-1)^n\, a_n\, n!}{\Gamma(n+2q)}\, L_n^{2q-1} (t)\, dt
\end{equation}
where
$$g(x) = \sum_{n=0}^\infty\, a_n\, (x-1)^n$$
and $\{a_{n}\}$ satisfies \eqref{vel-an}. Consider first the case $z\in \C \setminus [1,\infty)$, for which $|1-ze^{-t}|$ is bounded by a constant for all $t\in [0,\infty)$. Moreover, since 
$$
|e^{-t(x+1)}| \le  e^{-t} \qquad \forall\, x\in B
$$
we can argue as in the end of Proposition \ref{immersione} to write
$$(\QQ_{q,z}g)(x) = z\, \sum_{n=0}^\infty\, \frac{(-1)^n\, a_n\, n!}{\Gamma(n+2q)}\, \int_{0}^{\infty}\, \frac{e^{-t(x+1)}\, t^{2q-1}}{1-ze^{-t}}\, L_n^{2q-1} (t)\, dt$$
Since
$$\sum_{n=0}^\infty\, \sup_{x\in B} \left|  \frac{(-1)^n\, a_n\, n!}{\Gamma(n+2q)}\, \int_{0}^{\infty}\, \frac{e^{-t(x+1)}\, t^{2q-1}}{1-ze^{-t}}\, L_n^{2q-1} (t)\, dt \right| \le$$ $$\le  \sum_{n=0}^\infty\, \frac{|a_n|\, n!}{|\Gamma(n+2q)|}\, \left( \int_{0}^{\infty}\, \frac{e^{-t}\, t^{2\xi-1}}{|1-ze^{-t}|^{2}}\, dt \right)^{\frac 12}\, \|L_{n}^{2q-1}\|_{L^{2}(m_{q})} < \infty$$
as in \eqref{finita}, the operators $\QQ_{q,z}$ are bounded and to conclude we need to show that for any bounded domain $C\subset \{\Re(q)>0\}$ it holds 
$$\sum_{n=0}^\infty\, \sup_{q\in C} \sup_{x\in B} \left|  \frac{(-1)^n\, a_n\, n!}{\Gamma(n+2q)}\, \int_{0}^{\infty}\, \frac{e^{-t(x+1)}\, t^{2q-1}}{1-ze^{-t}}\, L_n^{2q-1} (t)\, dt \right| <\infty$$
From \eqref{norma-l2-lag} and arguing as in \eqref{finita}, we have
$$\sum_{n=0}^\infty\, \sup_{q\in C} \sup_{x\in B} \left|  \frac{(-1)^n\, a_n\, n!}{\Gamma(n+2q)}\, \int_{0}^{\infty}\, \frac{e^{-t(x+1)}\, t^{2q-1}}{1-ze^{-t}}\, L_n^{2q-1} (t)\, dt \right| \le$$ $$\le \sum_{n=0}^\infty\, \sup_{q\in C} \frac{|a_n|\, n!}{|\Gamma(n+2q)|}\, \left( \int_{0}^{\infty}\, \frac{e^{-t}\, t^{2\xi-1}}{|1-ze^{-t}|^{2}}\, dt \right)^{\frac 12}\, \|L_{n}^{2q-1}\|_{L^{2}(m_{q})}\le $$ $$\le \frac{1}{\min_{t\in \R^+} |1-ze^{-t}|}\, \sum_{n=0}^\infty\, \sup_{q\in C} \frac{|a_n|\, \Gamma(2\xi)}{|\Gamma(2\imath \eta)|\, \sqrt{\min\{1, 4\eta^2\}}}\, (n+2\xi)^{\xi+2} < \infty$$
in the case $C \cap \{\Im(q)=0\}=\emptyset$, and by using the value of $\|L_{n}^{2q-1}\|_{L^{2}(m_{q})}$ in the real case if $C\cap \{\Im(q)=0\}\not=\emptyset$, where we recall that $|\Gamma(2\imath \eta)|\, \eta$ is bounded at $\eta=0$, having the Gamma function a simple pole at 0. This shows that the result holds for $z\in \C \setminus [1,\infty)$.

Let us consider now the case $z=1$. We rewrite \eqref{int-rep-fin} as
$$
(\QQ_{q,1}g)(x) =  \int_{0}^{\infty}\, \frac{e^{-t(x+1)}\, t^{2q-1}}{1-e^{-t}}\, \phi_{q}(t)\, dt = $$ $$=\int_{0}^{\infty}\, \frac{e^{-t(x+1)}\, t^{2q-1}}{1-e^{-t}}\, \phi_{q}(0)\, dt + \int_{0}^{\infty}\, \frac{e^{-t(x+1)}\, t^{2q-1}\, t}{1-e^{-t}}\, \frac{\phi_{q}(t)-\phi_{q}(0)}{t}\, dt=
$$
$$
= \Phi(1,2q,x+1)\, g(0) + \int_{0}^{\infty}\, \frac{e^{-t(x+1)}\, t^{2q-1}}{(1-e^{-t})/t}\, \sum_{n=0}^\infty\, \frac{(-1)^n\, a_n\, n!}{\Gamma(n+2q)}\, \frac{L_n^{2q-1} (t)-L_n^{2q-1} (0)}{t}\, dt
$$
where $\Phi(z,2q,x+1)$ denotes the Lerch zeta function (see Remark \ref{rem:lerch}). For the second term we show that we can repeat the same argument as above. Indeed $(1-e^{-t})/t$ is a bounded function on $[0,\infty)$ and for any $\delta>0$
$$
\left| \frac{L_n^{2q-1} (t)-L_n^{2q-1} (0)}{t} \right| \le \left\{ \begin{array}{ll} \frac 1\delta \sqrt{|L_n^{2q-1} (t)|^{2}+|L_n^{2q-1} (0)|^{2}} & \text{if $t\ge \delta$}\\[0.2cm] \frac 1\delta \frac{|\Gamma(n+2q)|}{n!}\, (1+\delta)^{n} & \text{if $t\le \delta$} \end{array} \right.
$$
from which it follows that
$$
\left\|\frac{L_n^{2q-1} (t)-L_n^{2q-1} (0)}{t}\right\|_{L^{2}(m_{q})} \le $$ $$ \le \frac 1\delta \left( \| L_n^{2q-1}\|_{L^{2}(m_{q})} + \sqrt{\Gamma(2\xi)} |L_n^{2q-1} (0)| + \sqrt{\Gamma(2\xi)} |\Gamma(2q) L_n^{2q-1} (0)|\, (1+\delta)^{n} \right)
$$
Finally, using \eqref{finita} and $L_n^{2q-1} (0) = \frac{\Gamma(n+2q)}{n!}$, by choosing $\delta$ small enough such that
$$
\limsup_{n\to\infty}\, \sqrt[n]{|a_{n}|}\, (1+\delta) <1,
$$
which is possible by \eqref{vel-an}, we obtain
$$
\sum_{n=0}^\infty\, \frac{|a_n|\, n!}{|\Gamma(n+2q)|}\, \left\|\frac{L_n^{2q-1} (t)-L_n^{2q-1} (0)}{t}\right\|_{L^{2}(m_{q})} <\infty
$$
and from this we can prove as above that
$$
q\to \int_{0}^{\infty}\, \frac{e^{-t(x+1)}\, t^{2q-1}}{(1-e^{-t})/t}\, \sum_{n=0}^\infty\, \frac{(-1)^n\, a_n\, n!}{\Gamma(n+2q)}\, \frac{L_n^{2q-1} (t)-L_n^{2q-1} (0)}{t}\, dt
$$
is analytic on $\Re(q)>0$. The proof follows from well-known properties of the Lerch zeta function $\Phi(1,2q,x+1)$.
\end{proof}

\noindent
Moreover we recall that
\begin{theorem}[\cite{BGI}] \label{spettro-N}
For $\Re(q)>0$, the operators $\PP_{1,q}$ and $N_{q}$ on the spaces $\HH^2_q$ and $L^{2}(m_{q})$, respectively, are of trace class.
\end{theorem}
\noindent
From this, Theorem \ref{isomorfismo-analitic} and Remark \ref{isom-l2}, where we defined the induced operators $Q_{q,z}$ \eqref{q-l2}, we immediately obtain
\begin{corollary} \label{q-traceclass}
The operators $\QQ_{q,z}$ on $A_{\infty}(D)$ are of trace class. Moreover, for $z\in \C \setminus [1,\infty)$ and $\Re(q)>0$, and for $\Re(q)> \frac 12$ if $z=1$, it holds
\begin{equation} \label{formula-traccia}
{\rm trace} (\QQ_{q,z}) = {\rm trace} (Q_{q,z}) =z\ \int_{0}^{\infty}\,
\frac{J_{2q-1}(2t)}{t^{2q-1}}\, (1-ze^{-t})^{-1}\, m_q(dt)
\end{equation}
using (\ref{enne}).
\end{corollary}

\noindent
Applying Fredholm theory \cite{Gr} to the operators $\QQ_{q,z}$, and putting together Theorems \ref{thm:lerch-zeta} and \ref{isomorfismo-analitic}, we conclude that

\begin{corollary} \label{esiste-det}
For $z\in \C \setminus [1,\infty)$ the Fredholm determinants $q\mapsto \det(1\pm \QQ_{q,z})$ are analytic functions in $\Re(q)>0$. For $\Re(q)>0$ the the Fredholm determinants $z\mapsto \det(1\pm \QQ_{q,z})$ are analytic functions in $z\in \C \setminus [1,\infty)$. For $z=1$ the determinants $q\mapsto \det(1\pm \QQ_{q,1})$ are analytic functions in $\Re(q)> \frac 12$ with a meromorphic extension to $\Re(q)>0$ with a simple pole at $q=\frac 1 2$. 
\end{corollary}

\noindent
Using Corollary \ref{esiste-det} we can express  $Z(q,z)$ and $\zeta(q,z)$ in terms of the Fredholm determinants $\text{det}(1\pm \QQ_{q,z})$. By (\ref{formula-traccia}), this is an easy generalisation of results from \cite[Section 4]{Is} and \cite{Ma1,Ma3}.  More precisely we have 
\begin{theorem} \label{main-zeta}
For $z\in \C \setminus [1,\infty)$ and $\Re(q)> 0$ one has
\begin{equation} \label{due-zeta}
Z(q,z) = \det \left[ (1-\QQ_{q,z}) (1+\QQ_{q,z}) \right]
\end{equation}
and
\begin{equation}\label{R-zeta}
\zeta (q,z) = (1-z)^{-1}\ \frac{\det (1+\QQ_{q+1,z})}{\det(1-\QQ_{q,z})}
\end{equation}
which are analytic, respectively meromorphic, functions. Moreover 
$$
(q,z) \mapsto \zeta(q,z) \, Z(q,z) = (1-z)^{-1}\ \det (1+\QQ_{q+1,z})\ \det (1+\QQ_{q,z})
$$
is analytic in $\{ z\in \C \setminus [1,\infty)\}\times \{q\in \C \,: \, \Re(q)>0\}$.\\
For $z=1$ the function $Z(q,1)$ satisfies (\ref{due-zeta}) and is analytic for $\Re(q)>\frac 1 2$. Moreover it can be continued to $\Re(q)>0$ as a meromorphic function with a simple pole at $q= \frac 1 2$.
\end{theorem}

\noindent
The case $z=1$ for the Ruelle zeta function $\zeta(q,z)$ is more 
delicate as is evident from the term $(1-z)^{-1}$ in (\ref{R-zeta}). However at the end of Section \ref{sect-ff} we show that it is possible to define $\zeta(q,1)$ for $\Re(q)>1$.

\begin{remark} \label{reale}
For $q$ real there will be a simple pole, respectively a simple zero, at $z= \frac{1}{\lambda(q)} \in \left( \frac 1 2, 1\right)$ with $\lambda(q)$ defined in (\ref{pressure}). More information on poles, respectively zeroes, for the zeta functions can be obtained by the study of the point spectrum of the operators $\PP^{\pm}_{q}$ in $L^{2}(m_{q})$. Indeed by Corollary \ref{relaz-autof}, eigenfunctions of $\PP^{+}_{q}$ in $L^{2}(m_{q})$ with eigenvalue $\frac 1 z \not\in (0,1]$ are in one-to-one correspondence with eigenfunctions of $\QQ_{q,z}$ in $\tilde \HH^{2}_{q}$ with eigenvalue $1$. For real $q$, the spectrum of $\PP^{+}_{q}$ in $L^{2}(m_{q})$ has  been studied numerically in \cite{Pre}. The results suggest that $\lambda(q) \in (1,2)$ for $q\in (0,1)$ is the only eigenvalue, with the rest of the spectrum given by the interval $[0,1]$ which is purely continuous (see also \cite{BGI}). Whereas there are no eigenvalues for $q\ge 1$ and the spectrum is purely continuous. This would imply that for $q\in (0,1)$ the function $z\mapsto \zeta(q,z)$ has only one simple pole at $z= \frac{1}{\lambda(q)}$ and has no poles for $q\ge 1$, and respectively the function $z\mapsto Z(q,z)$ has only one single zero at $z= \frac{1}{\lambda(q)}$ for $q\in (0,1)$ and no zeroes for $q\ge 1$. 
\end{remark}

\subsection{Remarks on the case $z=1$} \label{z=1}

We now give some remarks on the relations with the works of Mayer and Lewis-Zagier, who have studied the case $z=1$. We first mention that in \cite{Ma3} the meromorphic extension for $Z(q,1)$ that we have obtained in Theorem \ref{main-zeta} is given to all the complex $q$-plane. Moreover, (\ref{due-zeta}) and (\ref{R-zeta}) give an explicit connection between zeroes of $Z(q,z)$, respectively zeroes or poles of $\zeta(q,z)$, and the existence of eigenfunctions $g\in A_{\infty}(D)$ for $\QQ_{q,z}$ with eigenvalue $\lambda_{Q}=\pm 1$. By Corollary \ref{relaz-autof} this turns out to be a connection with eigenfunctions of $\PP_{q}^{\pm}$ with eigenvalues $\lambda_{P} = \frac 1 z$. These connections have been proved in \cite{Ma3} for the operators $\QQ_{q,1}$, and in \cite{LeZa} for the operators $\PP_{q}^{\pm}$ with $\lambda_{P} = 1$ (see also Proposition \ref{proprieta-farey}).

Finally the characterisation of the eigenfunctions for $\PP^{\pm}_{q}$ given in Corollary \ref{autofunz} together with results from \cite{LeZa} lead to a new proof of the following result from \cite{efrat}
\begin{theorem} \label{efrat}
Even, respectively odd, spectral zeroes of $Z(q,1)$ correspond to eigenfunctions of $\PP^{+}$, respectively $\PP^{-}$. Moreover the zeroes $q$ in $\Re(q)>0$ with the Riemann zeta function $\zeta_{R}(2q)=0$ correspond to eigenfunctions of $\PP^{+}_{q}$, hence of $\QQ_{q,1}$ with eigenvalue $\lambda_{Q}= 1$.
\end{theorem}

\begin{proof} By Corollary \ref{relaz-autof} and Theorem \ref{main-zeta}, the zeroes of $Z(q,1)$ correspond to eigenfunctions of $\PP^{\pm}_{q}$ with eigenvalue $\lambda_{P}=1$ and term $c=0$ in (\ref{forma-autofunz}). From Corollary \ref{autofunz} it also follows that for eigenfunctions of $\PP^{-}_{q}$ it holds $b=0$. Hence Corollary \ref{autofunz} and (\ref{lewis}) imply that
$$
\PP^{-}_{q} f = f \ \ \text{with $c=0$} \  \ \Rightarrow \ f(x) = \left\{ \begin{array}{ll} O(1) & x\to 0 \\[0.2cm] O(x^{-2\Re(q)}) & x\to \infty \end{array} \right.
$$
For eigenfunctions of $\PP^{+}_{q}$ instead two cases are possible, $b=0$ and $b\not= 0$. We have
$$
\PP^{+}_{q} f = f \ \ \text{with $c=b=0$} \  \ \Rightarrow \ f(x) = \left\{ \begin{array}{ll} O(1) & x\to 0 \\[0.2cm] O(x^{-2\Re(q)}) & x\to \infty \end{array} \right.
$$
$$
\PP^{+}_{q} f = f \ \ \text{with $c=0$, $b\not= 0$} \  \ \Rightarrow \ f(x) = \left\{ \begin{array}{ll} \sim \frac{b}{2q-1}\, \frac 1x & x\to 0 \\[0.2cm] \sim \frac{b}{2q-1}\, x^{1-2\Re(q)} & x\to \infty \end{array} \right.
$$
Applying Theorem 2 and the subsequent Corollary from \cite{LeZa} it follows that all eigenfunctions of $\PP^{-}_{q}$ with $c=0$ are period functions associated to cusp forms or Maass wave forms. Moreover from Proposition \ref{proprieta-farey}-(ii) it follows that they are odd period functions, in the sense of \cite{LeZa}, hence are associated to odd cusp forms. The same holds for all eigenfunction of $\PP^{+}_{q}$ with $c=b=0$ which are associated to even cusp forms. This concludes the proof for the spectral zeroes of $Z(q,1)$. The previous argument also implies that the other zeroes of the Selberg zeta function with $\Re(q)>0$ necessarily correspond to eigenfunctions of $\PP^{+}_{q}$ with $c=0$ and $b\not= 0$. For example it is well known that the zero at $q=1$ corresponds to the eigenfunction $f(x) = \frac 1x$ of $\PP^{+}_{1}$ and the zeroes satisfying $\zeta_{R}(2q)=0$ correspond to the eigenfunctions $f^{+}_{q}(x)$ defined in (\ref{zagier-f}).  
\end{proof}

\section{Connections with Farey fractions} \label{sect-ff}

We now use (\ref{due-zeta}) and (\ref{R-zeta}) to give expressions for the generalised Selberg and Ruelle zeta functions as exponentials of Dirichlet series. By  (\ref{due-zeta}), (\ref{R-zeta}) and Theorem \ref{q-vs-p} we can write for $z\in \C \setminus [1,\infty)$ and $\Re(q)> 0$
\begin{equation} \label{tre-zeta}
Z(q,z) = \det \left[ 1- z\, \PP_{1,q} \, \left(1-z\, \PP_{0,q} \right)^{-1} \right]  \det \left[ 1+ z\, \PP_{1,q} \, \left(1-z\, \PP_{0,q} \right)^{-1} \right]
\end{equation}
\begin{equation} \label{tre-R-zeta}
\zeta(q,z) = (1-z)^{-1}\ \frac{\det \left[ 1+ z\, \PP_{1,q+1} \, \left(1-z\, \PP_{0,q+1} \right)^{-1} \right]}{\det \left[ 1- z\, \PP_{1,q} \, \left(1-z\, \PP_{0,q} \right)^{-1} \right]}
\end{equation}
The existence of the right-hand sides can be justified as in \cite{Rugh}. Expression (\ref{tre-zeta}) can be given also for $Z(q,1)$ and it follows from the uniform convergence of the series representation of $\QQ_{q,z}$ on $A_{\infty}(D)$ given in Theorem \ref{gauss-map}. We now perform a formal calculation which gives a connection between the zeta functions $\zeta(q,z)$ and $Z(q,z)$ and the Farey fractions. 

\noindent
Let us consider the matrices
$$
\phi_0=\left(
  \begin{array}{cc}
  1 & 0\\
  1 & 1
  \end{array}\right) \qquad
  \phi_1=\left(
  \begin{array}{cc}
  0 & 1\\
  1 & 1
  \end{array}\right)
$$
as elements of the group $GL(2,\Z)$ acting on $\C$ as M\"obius
transformations, see \cite{A2}. Then
$(\PP_{i,q} f)(x) = \left( \phi'_i(x) \right)^q f \left(
\phi_i(x) \right)$.
Thus, each term in the expansion of $(\PP_q^\pm)^n$ can be
represented in terms of a product of the matrices $\phi_i$. At the
same time, the Farey fractions can be represented by means of a
subgroup of $SL(2,\Z)$ with generators
$$
L=\left(
  \begin{array}{cc}
  1 & 0\\
  1 & 1
  \end{array}\right) \quad \mbox{and} \quad
R=\left(
\begin{array}{cc}
  1 & 1\\
  0 & 1
  \end{array}\right)
$$

\begin{lemma} \label{lem-expansion}
Expanding
$$(\PP_q^+)^n + (\PP_q^-)^n - 2(\PP_{0,q})^n$$ one obtains
$2(2^{n-1}-1)$ terms which are twice all the possible combinations
of $n$ factors involving $L$ and $R$ and starting with $L$,
without the term $L^n$.
\end{lemma}

\begin{proof} 
The only products which do not cancel out
are those where $\PP_{1,q}$ appears an even numbers of times
(counted twice). On the other hand we point out that if $K=\left(
  \begin{array}{cc}
  0 & 1\\
  1 & 0
  \end{array}\right)$
then $K^2=Id$, $L=\phi_0$ and $LK=KR=\phi_1$. We thus have
$$
\phi_1\, {\underbrace {\phi_0\dots \phi_0}_{n}}\;\phi_1 = LK
\,{\underbrace {L\dots L}_{n}}\;KR = L \,{\underbrace {R\dots
R}_{n}}\;R
$$
The thesis now easily follows.
\end{proof}

\begin{proposition} \label{traccia-matrici}
Let $A$ be the matrix corresponding to the term
$$(\PP_{0,q})^{n_1} (\PP_{1,q})^{n_2} \cdots
(\PP_{i,q})^{n_l}$$ with $\sum_{j=1}^l n_j =n>1$ and
$i=(1+(-1)^l)/2$. Then, setting $T:=\mbox{\rm trace}(A)$, we have
$T>2$ and
$$\mbox{\rm trace}\left[(\PP_{0,q})^{n_1} (\PP_{1,q})^{n_2} \cdots
(\PP_{i,q})^{n_l} \right] = \frac{1}{\sqrt{T^2-4}} \left(
\frac{2}{T+\sqrt{T^2-4}} \right)^{2q -1}$$
\end{proposition}

\begin{proof} 
The proof of the first assertion amounts to
a straightforward verification.  Let ${\mathcal V}$ be the
composition operator acting as
$({\mathcal V} f)(x) = \varphi(x) f(\psi(x))$
If $\psi(x)$ is holomorphic in a disk and has there a unique fixed
point $\bar x$ with $|\psi'(\bar x)|<1$ then ${\mathcal V}$ is of
the trace-class with $\mbox{trace}({\mathcal V}) = \frac{\varphi(\bar x)}{1-\psi'(\bar
x)}.$
The thesis follows by applying this relation to the operators
under consideration.
\end{proof}

\noindent
We are now going to make use of the correspondence
between products of matrices $L$ and $R$ and the Farey fractions.
Let us consider the Farey tree $\FF$. We recall that every rational number ${a\over b} \in (0,1)$
appears exactly once in $\FF$. One may therefore identify ${a\over b}$
with the path on $\FF$ which reaches it starting from the root
node ${1\over 2}$ (first row), which in turn can be encoded as a matrix
product in the following way: first recall that every rational
number ${a\over b} \in (0,1)$ has a unique finite continued fraction
expansion ${a\over b} =[a_1, \dots, a_k]$ with $a_k>1$, and one may
define the rank of ${a\over b}$ as
$$
\frac a b =[a_1,\dots ,a_k]\; \Longrightarrow \;{\rm
rank}\left(\frac a b \right)=\sum_{i=1}^ka_i -1
$$ 
It turns out
that ${a\over b}$ has rank $n$ if and only if it belongs to $\FF_{n+1}
\setminus \FF_n$. Furthermore, we may uniquely decompose $\frac a
b$ as
$$\frac a b =\frac{a'+a''}{b'+b''} \quad \mbox{with} \quad
 b'a-a'b=ba''-ab''=b'a''-a'b''=1$$
The neighbours $\frac{a'}{b'}$ and $\frac{a''}{b''}$ are the
\emph{parents} of $\frac a b$ in $\FF$ and we may accordingly
identify
$$\frac a b \simeq \left(
\begin{array}{cc}
  a'' & a'\\
  b'' & b'
\end{array}\right) \in \ZZ$$
where
$$\ZZ:= \set{
\left(
\begin{array}{cc}
a & b\\
c & d
\end{array}
\right) \in SL(2,\Z) \ : \  0< a \le c, \, 0\le b< d }$$ Clearly
$\frac 1 2 \simeq L$ and, more generally,
\begin{equation} \label{frama}
\frac a b \simeq L\prod_i M_i
\end{equation}
where the number of terms in the product $L\prod_i M_i$ is equal
to ${\rm rank} ({a\over b})$ and $M_i=L$ or $M_i=R$ according to whether
the $i$-th turn, along the descending path in $\FF$ which starts
from the root node $\frac 1 2$ and reaches ${a\over b}$, goes to the left
or to the right.
Using (\ref{frama}) one may then define a map $T: \FF \to \N$ as
\begin{equation}\label{map-tr}
T\left( \frac a b \right) := \mbox{trace}(L\prod_i M_i)
\end{equation}
Note moreover that the set
$$
\tilde \FF_n:= \FF_{n+1}\setminus \FF_{n}=  \set{ \frac a b \in \FF \ : \ {\rm
rank}\left(\frac a b \right) = n }
$$
has $2^{n-1}$ elements which are in a one-to-one correspondence
with the (equal pairs of) elements in the expansion dealt with in
Lemma \ref{lem-expansion} plus $ \set{ \frac{1}{n+1}}$. We now  obtain an expression for the $Z(q,z)$ and $\zeta(q,z)$ as exponentials of power series whose coefficients are computed along lines of $\FF$.

\begin{theorem} \label{formula-selberg}
For $\Re(q) >1$ and $|z|\le 1$ the two-variable zeta functions $Z(q,z)$ and $\zeta(q,z)$ can be written as
$$
Z(q,z)= \exp \left( - \sum_{n\ge 2} \frac{z^n}{n}\ \Lambda_n(q)\right) \quad , 
\quad \zeta (q,z) = \exp \left( \sum_{n\ge 1} \frac{z^n}{n}\ \Xi_n(q)
\right)
$$ 
with
$$
\Lambda_n(q):=\sum_{\frac a b \in \tilde \FF_n \setminus  \set{ \frac{1}{n+1}}}
\frac{2}{\sqrt{T^2 (\frac a b)-4}} \left( \frac{2}{T(\frac a
b)+\sqrt{T^2(\frac a b)-4}} \right)^{2q -1}
$$
$$
\Xi_n(q)=\sum_{\frac a b \in \tilde \FF_n}\left[ \left( \frac{2}{T(\frac a b)+
\sqrt{T^2(\frac a b)-4}} \right)^{2q}+
 \left( \frac{2}{T(\frac a b)+\sqrt{T^2(\frac a b)+4}} \right)^{2q}\right]
$$ 
where the map $T$ is defined in (\ref{frama})-(\ref{map-tr}).
\end{theorem}

\begin{proof} We first obtain the expansion for $Z(q,z)$.
A formal manipulation of (\ref{tre-zeta}) gives 
\begin{equation} \label{quattro-zeta}
Z(q,z) =\det [(1-z\PP_q^+)(1-z\PP_q^-)(1-z\PP_{0,q})^{-2}]
\end{equation}
which is well defined by Lemma \ref{lem-expansion}. 

\noindent We now only have to prove that $\sum_{n\ge 2} \frac{z^n}{n}\ \Lambda_n(q)$ converge for $\Re(q) >1$ and $|z|\le 1$. Since then the assertion follows from
(\ref{quattro-zeta}), Proposition \ref{traccia-matrici} and the definition of the map $T$ in (\ref{frama})-(\ref{map-tr}). Set
$$
\gamma(k,n):=\# \set{ \frac a b \in \tilde \FF_n \setminus  \set{ \frac{1}{n+1}}\ :\
T\left( \frac a b \right)=k}
$$
and let ${\mathcal M}$ be the free multiplicative monoid generated by
the matrices $L$ and $R$. The function
$$
\Psi (k)= \# \{ X\in {\mathcal M} \ :\ {\rm trace} (X)\le  k\}
$$
has been recently studied in the literature and the asymptotic behaviour
$$
\Psi (k)= {k^2 \log k \over \zeta_{R}(2)} + O(k^2)
$$
has been found (see \cite{Bo,KOPS}). Let us decompose $\Psi(k)$ as
$\Psi(k)=\Psi_L (k)+\Psi_R (k)$ where $\Psi_L (k)$ (resp. $\Psi_R
(k)$) is obtained restricting to the elements of ${\mathcal M}$
which start with $L$ (respectively $R$). Note that
$$\Psi_L (k)=\sum_{j\le k} \sum_{n=2}^{j-1} \gamma(j,n)$$
Moreover, using $R^kK = K L^k$ one easily realises that if
$$  {a\over b}\simeq \textstyle L\prod_i M_i= \left(
\begin{array}{cc}
a'' & a'\\
b'' & b'
\end{array}\right)
$$
then, setting ${\overline M}_i=L$ if $M_i=R$ and viceversa, we
have
$$ {b\over a}\simeq \textstyle R\prod_i {\overline M}_i=\left(
\begin{array}{cc}
b' & b''\\
a' & a''
\end{array}\right) 
$$
Therefore $T\left( \frac a b \right)=T\left( \frac b a \right)$ and
$$\sum_{j\le k} \sum_{n=2}^{j-1} \gamma(j,n) =
\frac{k^2 \log k}{2 \zeta_{R}(2)} + O(k^2)$$ This implies that if
$$\alpha(k) := \sum_{n=2}^{k-1} \frac{\gamma(k,n)}{n}$$
then
$$\sum_{j\leq k}\alpha (j) =  O(k^2\log k)$$ hence
$$\sum_{n\ge 2} \frac{1}{n}\ \Lambda_n(q) =
\sum_{k=3}^\infty \ \frac{2\, \alpha(k)} {\sqrt{k^2-4}}
\left(\frac{2}{k+\sqrt{k^2-4}} \right)^{2q -1}$$ converges
absolutely for $\Re(q) >1$. The case $|z|<1$ easily follows.

\noindent
To obtain the expansion for $\zeta(q,z)$, we again start
from the following formal manipulation of (\ref{tre-R-zeta})
\begin{equation} \label{quattro-R-zeta}
\zeta (q,z)= \frac{\det [ (1-z\, \PP_{q+1}^{-}) (1-z\PP_{q}^{+})^{-1} (1-z\,
\PP_{0,q+1})^{-1} (1-z\, \PP_{0,q}) ] }{(1-z)}
\end{equation}
Now, given $\epsilon >
0$ let us consider the perturbation $M_{q,\epsilon}$ of the
operator $M$ in (\ref{emme}) acting as
$$(M_{q,\epsilon} \,\varphi)(t) = \frac{e^{-(\frac{1-\epsilon}{1+\epsilon})
t}}{(1+\eps)^q}\ \varphi \left(\frac{t}{1+\eps} \right)$$ 
Reasoning as in the proof of \cite[Proposition 4.5]{GI} one easily sees that
for all $\epsilon >0$ the operator $M_{q,\epsilon}$ is of trace class
on $L^2(m_q)$ and its spectrum is given by $\sigma
(M_{q,\epsilon})= \set{ (1+\epsilon)^{-q-k} }_{k\ge 0}$, each eigenvalue
being simple. Therefore
$$ {\rm trace} \, (M_{q,\epsilon}^n - M_{q+1,\epsilon}^n) = (1+\epsilon)^{-q-n}$$
Let moreover $\PP_{0,q}^{\eps}$ be the corresponding operator on $\HH^2_q$. A short calculation then gives
$$\lim_{\eps \to 0^+} \det [ (1-z\,\PP_{0,q+1}^{\eps})^{-1}
(1-z\, \PP_{0,q}^{\eps})] = \lim_{\epsilon \to 0^+} \left(
1-\frac{z}{1+\epsilon} \right)^{(1+\epsilon)^{-q}} = 1-z$$
and therefore, letting $\PP_{q,\epsilon}^{\pm}:=\PP_{0,q}^{\eps} \pm \PP_{1,q}$,
\begin{equation} \label{ultima}
\zeta (q,z)= \lim_{\eps \to 0^+} {\det [ (1-z\, \PP_{q+1, \epsilon}^{-})  (1-z\,
\PP_{q,\epsilon }^{+})^{-1} ] }
\end{equation}
Starting from this expression along with \cite[Corollary 3.11]{DEIK} and proceeding as above one readily obtains the expansion for $\zeta(q,z)$. 
\end{proof}

\noindent
We remark that from Theorem \ref{formula-selberg}, in particular 
from (\ref{ultima}), we obtain a definition for $\zeta(q,1)$ for $\Re(q)>1$, which is not immediate from (\ref{R-zeta}). Hence we can extend Theorem 
\ref{main-zeta} to the case $z=1$, in particular the function $q 
\mapsto \zeta(q,1) Z(q,1)$ turns out to be analytic for $\Re(q)>1$.

\section*{Acknowledgements.}
C.B. is partially supported by King Saud University, Riyadh, Saudi Arabia.


\begin{thebibliography}{99}

\bibitem{A1} T. Apostol, ``Introduction to Analytic Number Theory'',
Undergraduate Texts in Mathematics, Springer, New York, 1976.

\bibitem{A2} T. Apostol,
``Modular Functions and Dirichlet Series in Number Theory'',
Graduate Texts in Mathematics 41, Springer, New York, 1976.

\bibitem{Bal} V. Baladi,
``Positive Transfer Operators and Decay of Correlations'',
Advanced Series in Nonlinear Dynamics 16, World Scientific, River Edge, 2000.

\bibitem{Bo} F.P. Boca,
{\it  Products of matrices {\small $\left(
\begin{array}{cc}
  1 & 1\\
  0 & 1
  \end{array}\right)$} and {\small $\left(
\begin{array}{cc}
  1 & 0\\
  1 & 1
\end{array}\right)$} and the distribution of reduced quadratic irrationals},
J. Reine Angew. Math. {\bf 606} (2007) 149--165.

\bibitem{BGI} C. Bonanno, S. Graffi, S. Isola,
\emph{Spectral analysis of transfer operators associated to Farey
fractions}, Atti Accad. Naz. Lincei Cl. Sci. Fis. Mat. Natur.
Rend. Lincei (9) Mat. Appl. \textbf{19} (2008) 1--23.

\bibitem{CM1} C-H. Chang, D.H. Mayer, \emph{The period function of the nonholomorphic Eisenstein series for ${\rm PSL}(2,Z)$}, Math. Phys. Electron. J. \textbf{4} (1998), paper 6.

\bibitem{CM2} C-H. Chang, D.H. Mayer, \emph{The transfer operator approach to Selberg's zeta function and modular and Maass wave forms for ${\rm PSL}(2,Z)$}, in A. Hejhal, M. Gutzwiller et al. eds, ``Emerging applications of number theory'', 73--141, IMA Vol. Math. Appl., 109, Springer, New York, 1999.

\bibitem{DEIK} M. Degli Esposti, S. Isola, A. Knauf,
\emph{Generalized Farey trees, transfer operators and phase
transitions}, Comm. Math. Phys. \textbf{275} (2007) 297--329.

\bibitem{efrat} I. Efrat, \emph{Dynamics of the continued fraction map and the spectral theory of $SL(2,{\mathbf{Z}})$}, Invent. Math. \textbf{114} (1993), 207--218.

\bibitem{E} A. Erd\`ely et al.,
``Higher transcendental functions'', Bateman manuscript
project, vols. I-III, McGraw-Hill, New York, 1953-1955.

\bibitem{GI} M. Giampieri, S. Isola,
\emph{A one-parameter family of analytic Markov maps with an
intermittency transition}, Discrete Contin. Dyn. Syst. \textbf{12}
(2005) 115--136.

\bibitem{Gr} A. Grothendieck,
{\it La th\`eorie de Fredholm}, Bull. Soc. Math. France \textbf{84} (1956) 319--384.

\bibitem{Is} S. Isola,
\emph{On the spectrum of Farey and Gauss maps}, Nonlinearity
\textbf{15} (2002) 1521--1539.

\bibitem{KOPS} J. Kallies, A. \"Ozluk, M. Peter, C. Snyder,
{\it On asymptotic properties of a number theoretic function
arising out of a spin chain model in statistical mechanics},
Commun. Math. Phys. \textbf{222} (2001) 9--43.

\bibitem{Lag} J.C. Lagarias, \emph{Number theory zeta functions and dynamical zeta functions}, in: T. Branson (Ed.), ``Spectral Problems in Geometry and Arithmetic'',   Contemp. Math. Vol. 237, Amer. Math. Soc., Providence, RI, 1999, pp. 45--86.

\bibitem{Le} J.B. Lewis,
\emph{Spaces of holomorphic functions equivalent to  even Maass
cusp forms}, Invent. Math. \textbf{127} (1997) 271--306.

\bibitem{LeZa} J.B. Lewis, D. Zagier,
\emph{Period functions for Maass wave forms. I}, Ann. Math.
\textbf{153} (2001) 191--258.

\bibitem{Ma1} D.H. Mayer, {\it On the thermodynamic formalism for
the Gauss map}, Commun. Math. Phys. {\bf 130} (1990) 311--333.

\bibitem{Ma2} D.H. Mayer,
{\it Continued fractions and related transformations}, in: T. Bedford, M. Keane, C. Series (Ed.), ``Ergodic Theory, Symbolic Dynamics and  Hyperbolic Spaces'', Oxford University Press, New York, 1991, pp. 175--222.

\bibitem{Ma3} D.H. Mayer,
{\it The thermodynamic formalism approach to Selberg's zeta
function for $PSL(2\Z)$}, Bull. Amer. Math. Soc. {\bf 25} (1991) 55--60.

\bibitem{Pre} T. Prellberg, {\it Towards a complete determination of the spectrum of a transfer operator associated with intermittency}, J. Phys. A {\bf 36} (2003) 2455--2461.

\bibitem{PS} T. Prellberg, J. Slawny, {\it Maps of intervals with
indifferent fixed points: thermodynamic formalism and phase transitions},
J. Stat. Phys. {\bf 66} (1992) 503--514.

\bibitem{Rugh} H.H. Rugh, \emph{Intermittency and regularized Fredholm determinants}, Invent. math. {\bf 135} (1999) 1--24.

\bibitem{Rue} D. Ruelle, \emph{Zeta-functions for expanding maps and Anosov flows}, Invent. math. {\bf 34} (1976) 231--242.

\bibitem{Se} C. Series, \emph{The modular surface and continued fractions}, J. London. Math. Soc. {\bf 31} (1985) 69--80.

\bibitem{lag} H.M. Srivastava, H.A. Mavromatis, R.S. Alassar, \emph{Remarks on some associated Laguerre results}, Appl. Math. Lett. \textbf{16} (2003) 1131--1136

\end{thebibliography}
\end{document}